\documentclass[12pt,reqno]{amsart}

\usepackage{amsfonts}
\usepackage{amsthm}
\usepackage{amsmath}
\usepackage{amssymb}
\usepackage{mathtools} 
\usepackage{xcolor}
\usepackage{soul} 
\usepackage{mathrsfs}

\usepackage[margin=0.8in]{geometry}
\usepackage{amsmath,amssymb,amsthm,graphicx,amsxtra, setspace}
\usepackage[utf8]{inputenc}
\usepackage{mathrsfs}
\usepackage{eucal}
\usepackage{hyperref}
\usepackage{upgreek}
\usepackage{mathtools}
\usepackage{multirow}
\usepackage{mathabx}
\usepackage{graphicx,type1cm,eso-pic,color}
\allowdisplaybreaks

\usepackage[pagewise]{lineno}

\DeclareMathAlphabet{\mathpzc}{OT1}{pzc}{m}{it}

\usepackage[cyr]{aeguill}

\colorlet{darkblue}{blue!50!black}

\hypersetup{
	colorlinks,%
	citecolor=blue,%
	filecolor=red,%
	linkcolor=red,%
	urlcolor=blue,%
	pdfnewwindow=true,%
	pdfstartview={FitH}
}

\newtheorem{theorem}{Theorem}[section]
\newtheorem{lemma}[theorem]{Lemma}

\newtheorem{corollary}[theorem]{Corollary}
\newtheorem{definition}[theorem]{Definition}

\newtheorem{remark}[theorem]{Remark}

\let\originalleft\left
\let\originalright\right
\renewcommand{\left}{\mathopen{}\mathclose\bgroup\originalleft}
\renewcommand{\right}{\aftergroup\egroup\originalright}

\renewcommand{\d}{\/\mathrm{d}\/}

\def\w{\textbf{W}^{\varepsilon}_{{\theta}^{\varepsilon}}}

\def\e{\varepsilon}

\def\S{\mathrm{S}}
\def\T{\mathbb{T}}
\def\L{\mathbb{L}}
\def\A{\mathrm{A}}
\def\I{\mathrm{I}}

\def\C{\mathrm{C}}
\def\f{\boldsymbol{f}}

\def\B{\mathrm{B}}
\def\D{\mathrm{D}}
\def\y{\boldsymbol{y}}

\def\x{\boldsymbol{x}}

\def\p{\boldsymbol{p}}

\def\z{\boldsymbol{z}}
\def\v{\boldsymbol{v}}
\def\V{\mathbb{v}}
\def\w{\boldsymbol{w}}
\def\W{\mathrm{W}}

\def\Q{\mathrm{Q}}

\def\N{\mathbb{N}}

\def\V{\mathbb{V}}
\def\wi{\widetilde}
\def\Q{\mathrm{Q}}
\def\u{\mathrm{U}}

\def\u{\boldsymbol{u}}
\def\H{\mathbb{H}}

\newcommand{\R}{\mathbb{R}}

\renewcommand{\d}{\/\mathrm{d}\/}

\numberwithin{equation}{section}

\newcommand{\Addresses}{{
		\footnote{
			
			\noindent \textsuperscript{1}Center for Mathematics and Applications (NovaMath),  NOVA	SST,	Portugal. \par\nopagebreak
			\noindent \textsuperscript{2}Department of Mathematics, Indian Institute of Technology Roorkee-IIT Roorkee,
			Haridwar Highway, Roorkee, Uttarakhand 247667, India.\par\nopagebreak
			\noindent  \textit{e-mail:} \texttt{Manil T. Mohan: maniltmohan@ma.iitr.ac.in, maniltmohan@gmail.com.}
			
			\textit{e-mail:} \texttt{Kush Kinra: kushkinra@gmail.com, k.kinra@fct.unl.pt.}
			
			\noindent \textsuperscript{*}Corresponding author.

			\textit{Key words:} Convective Brinkman-Forchheimer  equations, backward uniqueness, approximate controllability, Dirichlet's quotient, Lipschitz deviation, Gaussian noise.
			
			Mathematics Subject Classification (2010): Primary 35Q30, 76D05; Secondary 35R60, 37L30, 76B75.

}}}

\begin{document}

		\title[Backward uniqueness of CBF equations]{Backward uniqueness of 2D and 3D  convective Brinkman-Forchheimer  equations and its applications \Addresses}
	\author[K. Kinra and M. T. Mohan]{Kush Kinra\textsuperscript{1} and Manil T. Mohan\textsuperscript{2*}}

	\maketitle
	

	\begin{abstract}
		 In this work, we consider the two- and three-dimensional convective Brinkman-Forchheimer (CBF)  equations  (or damped Navier--Stokes equations) on a torus $\mathbb{T}^d,$ $d\in\{2,3\}$:
	$$	\frac{\partial \boldsymbol{u}}{\partial t}-\mu \Delta\boldsymbol{u}+(\boldsymbol{u}\cdot\nabla)\boldsymbol{u}+\alpha\boldsymbol{u}+\beta|\boldsymbol{u}|^{r-1}\boldsymbol{u}+\nabla p=\boldsymbol{f}, \ \nabla\cdot\boldsymbol{u}=0,$$  where $\mu,\alpha,\beta>0$ and $r\in[1,\infty)$ is the absorption exponent.  For $d=2,r\in[1,\infty)$ and $d=3,r\in[3,\infty)$ ($2\beta\mu\geq 1$ for $d=r=3$), we first show the backward uniqueness of deterministic CBF equations by exploiting  the logarithmic convexity property and the global solvability results available in the literature. As a direct consequence of the backward uniqueness result, 	we first derive the approximate controllability with respect to the initial data  (viewed as a start controller). Secondly, we apply  the backward uniqueness results  in the attractor theory to show the zero Lipschitz deviation of the global  attractors  for 2D and 3D CBF equations. By an application of log-Lipschitz regularity,  we  prove the uniqueness of Lagrangian trajectories in 2D and 3D CBF flows  and the continuity of Lagrangian trajectories with respect to the Eulerian initial data. Finally, we consider the stochastic CBF equations with a linear multiplicative Gaussian noise. For $d=2,r\in[1,\infty)$ and $d=3,r\in[3,5]$ ($2\beta\mu\geq 1$ for $d=r=3$), we show the pathwise backward uniqueness  as well as approximate controllability  via starter controller results. 
	\end{abstract}

\section{Introduction}\label{Sec1}\setcounter{equation}{0}
Let $\mathfrak{A}(\cdot)$ be a bounded or unbounded linear operator on a Hilbert space $\mathcal{H}$  and ${\varphi}$ be a solution of the evolution equation $\frac{\d}{\d t}{\varphi}+\mathfrak{A}({\varphi})={0}$ such that $\varphi(T)=0$ for some $T>0$ (\cite{JMG}). We have the \emph{backward uniqueness property} if $\varphi(t)=0$ for all $0\leq t\leq T$. If $\mathfrak{A}$ is bounded from below, then the energy method helps to prove that  $\varphi(t)=0$ for $ t\geq T$. This method generally fails for $0\leq t\leq T$  since the Cauchy problem for $\varphi$ is usually an  ill-posed backward in time problem.  Therefore, backward uniqueness is substantially more difficult than forward uniqueness due to ill-posedness (in general) of the backward evolution problem (\cite{IK}). The backward uniqueness results  for parabolic problems based on logarithmic convexity method have been examined in \cite{SALN,CBLT,LEGA,JMG,IK,JLLBM,HO}. We explore an interesting connection of backward uniqueness and approximate controllability with respect to the initial data  (viewed as a start controller) for two- and three-dimensional convective Brinkman-Forchheimer (CBF)  equations  (or damped Navier--Stokes equations) in this work. Another motivation for examining the backward uniqueness of CBF equations is its amazing connection with the study of attractors and the long-term behavior of infinite-dimensional dynamical systems. Moreover, by using  log-Lipschitz regularity,  we  prove the uniqueness of Lagrangian trajectories in 2D and 3D CBF flows. We cannot solve CBF equations backwards (ill-posed problem), but one can show that the regular solutions enjoy the \emph{backward uniqueness property}. 
\subsection{The model}
The convective Brinkman-Forchheimer (CBF) equations describe the motion of incompressible fluid flows in a saturated porous medium.  We consider the following  CBF equations in  a $d$-dimensional torus $\mathbb{T}^{d}=\big(\R/\mathrm{L}\mathbb{Z}\big)^{d}$ ($d=2,3$):
\begin{equation}\label{1}
	\left\{
	\begin{aligned}
		\frac{\partial \u}{\partial t}-\mu \Delta\u+(\u\cdot\nabla)\u+\alpha\u+\beta|\u|^{r-1}\u+\nabla p&=\f, \ \text{ in } \ \mathbb{T}^{d}\times(0,T], \\ \nabla\cdot\u&=0, \ \text{ in } \ \mathbb{T}^{d}\times(0,T], \\
		\u(0)&=\u_0 \ \text{ in } \ \mathbb{T}^{d},
	\end{aligned}
	\right.
\end{equation}
where $\u(x,t):\mathbb{T}^{d}\times[0,T]\to\R^d$ represents the velocity field at time $t$ and position $x$, $p(x,t):\mathbb{T}^{d}\times[0,T]\to\R$ denotes the pressure field and $\f(x,t):\mathbb{T}^{d}\times[0,T]\to\R^d$ is an external forcing. Moreover, $\u(\cdot,\cdot)$, $p(\cdot,\cdot)$ and $\f(\cdot,\cdot)$ satisfy the following periodicity conditions:
\begin{align}\label{2}
	\u(x+\mathrm{L}e_{i},\cdot) = \u(x,\cdot), \ p(x+\mathrm{L}e_{i},\cdot) = p(x,\cdot) \ \text{ and } \ \f(x+\mathrm{L}e_{i},\cdot) = \f(x,\cdot),
\end{align}
for every $x\in\R^{d}$ and $i=1,\ldots,d,$ where $\{e_{1},\dots,e_{d}\}$ is the canonical basis of $\R^{d}.$ The positive constants $\mu,\alpha$ and $\beta$ denote the \emph{Brinkman coefficient} (effective viscosity), \emph{Darcy} (permeability of the porous medium) and \emph{Forchheimer} (proportional to the porosity of the material) coefficients, respectively. The absorption exponent $r\in[1,\infty)$ and  $r=3$ is known as the \emph{critical exponent}. The critical homogeneous CBF equations  have the same scaling as Navier--Stokes  equations (NSE)  only when $\alpha=0$ (\cite{KWH}).  In the literature,  the case $r<3$ is referred  as \emph{subcritical} and $r>3$ as \emph{supercritical} (or fast growing nonlinearities, \cite{KT2}). For the supercritical case, the diffusion ($-\Delta\u$) and damping ($|\u|^{r-1}\u$)  terms dominate the convective term $(\u\cdot\nabla)\u$, and one can expect global solvability results for the system \eqref{1}  (see \cite{SNA,KWH,KT2,MTT,MT1} and references therein).   Moreover, the system \eqref{1} is also referred as NSE modified by an absorption term (\cite{SNA}). The above model is accurate when the flow velocity is too large for Darcy's law to be valid alone, and apart from that the porosity is not too small (\cite{MTT}).  If one considers \eqref{1} with $\alpha=\beta=0$, then we obtain the classical NSE, which  describe the motion of viscous fluid substances, and if $\alpha, \beta>0$, then it can also be considered as  a damped NSE. 

The authors in \cite{MTT} considered the system \eqref{1} with an extra term $\widetilde{\beta}|\u|^{\widetilde{r}-1}\u$ to model a pumping, when $\widetilde{\beta}<0$ by opposition to the damping modeled through the term $\beta|\u|^{r-1}\u$ when $\beta>0$ (referred as Brinkman-Forchheimer extended Darcy (BFeD) model). For $\beta>0$ and $\widetilde{\beta}\in\R$, the existence of weak  solutions is obtained by  assuming that $r>\widetilde{r}\geq 1$, and the continuous dependence on the data as well as the existence of strong solutions were established for $r>3$. As we are working on the torus $\mathbb{T}^d$  and $\widetilde{r}<r$, by modifying  calculations suitably, the  results of this paper hold true for BFeD model also.

\subsection{Literature survey}
A good number of literature is available for the global solvability results of the system \eqref{1} (for instance, see \cite{SNA,KWH,KWH1,KT2,MTT,MT1}).  The existence of a global weak solution to the system \eqref{1}  is established in \cite{SNA} and its uniqueness (for $d=2$, $r\in[1,\infty)$ and for $d=3$, $r\in[3,\infty)$) in \cite{MTT,MT1}.  The Brinkman-Forchheimer equations with fast growing nonlinearities is considered in \cite{KT2} and the authors established the existence of regular dissipative solutions and global attractors for 3D CBF equations  for $r> 3$. The authors in \cite{CLF,KWH} proved that all weak solutions of the critical 3D CBF equations  satisfy the energy equality in  bounded as well as periodic domains.  For the global solvability and random dynamics of stochastic CBF equations, one may refer to \cite{KKMTM,MT4,MT2}, etc. Likewise 3D NSE, the global solvability of deterministic and stochastic subcritical  3D CBF equations ($2\beta\mu<1$ for critical also) is still an open problem.  

 For ordinary differential equations,  the backward uniqueness is equivalent to the forward uniqueness. Whereas, for stochastic differential equations, the backward uniqueness is closely related to the question of existence of a stochastic flow and the latter implies former (\cite{ZBMN}).  Similar to deterministic partial differential equations (PDEs), in the case of stochastic PDEs, the existence of a flow does not  imply the backward uniqueness. The main applications of backward uniqueness are in the long time behavior of the solutions (cf. \cite{ZBMN,JMG,IK,JCR5}, etc.) and control theory (\cite{VBMR,CFJP,SMEZ}, etc.). The backward uniqueness as well as unique continuation property for various equations, like  NSE, Kuramoto-Sivashinsky equations, nonlinear dissipative Schr\"odinger equation, etc. are established in \cite{JMG}.  By proving the backward uniqueness result, the asymptotic behavior for large times of solutions of linear stochastic PDEs of parabolic type is investigated in \cite{ZBMN}. The backward uniqueness property of the solution to 3D stochastic magnetohydrodynamic-$\alpha$ model driven by a linear multiplicative Gaussian noise is studied  in \cite{RZ}. The authors in \cite{VBMR} proved the backward uniqueness of solutions to stochastic semilinear parabolic equations as well as for the tamed NSE driven by a linear multiplicative Gaussian noise. They have provided applications to the approximate controllability of nonlinear stochastic parabolic equations with initial controllers are given. The backward uniqueness of 3D NSE and their applications have been explored in  \cite{LEGA,LEGA1,JMG,MIIK,RJK,LEP,GSVS,TT,TPT}  and references therein. The authors in  \cite{DHET,MP}  proved the backward uniqueness results for  3D NSE of  compressible flow and primitive equations, respectively. In this work, we establish the backward uniqueness result for deterministic and stochastic 2D and 3D CBF equations and investigate  their applications in control theory and attractor theory.

\subsection{Difficulties, approaches and novelties}
We consider  the CBF equations \eqref{1} in a $d$- dimensional torus only.  In the  torus $\mathbb{T}^d$ as well as on the whole space $\mathbb{R}^d$, the Helmholtz-Hodge projection $\mathcal{P}$ and $-\Delta$ is commute (\cite[Theorem 2.22]{JCR4}). Therefore the equality  (\cite[Lemma 2.1]{KWH})
\begin{align}\label{3}
	&\int_{\mathbb{T}^{d}}(-\Delta \boldsymbol{y}(x))\cdot|\boldsymbol{y}(x)|^{r-1}\boldsymbol{y}(x)\d x\nonumber\\&=\int_{\mathbb{T}^{d}}|\nabla \boldsymbol{y}(x)|^2|\boldsymbol{y}(x)|^{r-1}\d x+4\left[\frac{r-1}{(r+1)^2}\right]\int_{\mathbb{T}^{d}}|\nabla|\boldsymbol{y}(x)|^{\frac{r+1}{2}}|^2\d x,
\end{align}
is quite useful in obtaining the regularity results. The above equality may not be useful in  domains other than the whole space or a $d$-dimensional torus (see \cite{KWH,MT1} for a detailed discussion). For $\x\in\H$ and $\f\in\mathrm{L}^2(0,T;\H)$ (see Section \ref{Sec2} below for functional setting), using the above estimate, one can show that the weak solution $\u\in\C([0,T];\H)\cap\mathrm{L}^2(0,T;\V)\cap\mathrm{L}^{r+1}(0,T;\wi\L^{r+1})$ to the system \eqref{555} (see below) has the regularity $$\u\in\C((0,T];\V)\cap\mathrm{L}^2(\epsilon,T;\D(\A))\cap\mathrm{L}^{r+1}(\epsilon,T;\wi\L^{3(r+1)})$$ for any $\epsilon>0$ (see \eqref{16}-\eqref{18} below). Furthermore, for the supercritical case, if $\frac{\d\f}{\d t}\in\mathrm{L}^2(0,T;\V')$, one can show that $\frac{\d\u}{\d t}\in\mathrm{L}^{\infty}(\epsilon_1,T;\H)\cap\mathrm{L}^2(\epsilon_1,T;\V)$  and if $\frac{\d\f}{\d t}\in\mathrm{L}^2(0,T;\H)$,  then one gets $\frac{\d\u}{\d t}\in\mathrm{L}^{\infty}(\epsilon_1,T;\V)$, for any $0<\epsilon<\epsilon_1$. Note that in the stochastic case, we are unable to obtain the above regularity on the time derivative, so we must restrict ourselves to $r\in[3,5]$ in three dimensions.

We point out here that in the deterministic case, the results obtained in this work hold true for bounded domains as well. For $\x\in\H$ and $\f\in\mathrm{W}^{1,2}(0,T;\H)$ (so that $\f\in\C([0,T];\H)$), the authors in \cite{KT2} established that the weak solution of \eqref{555} with fast growing nonlinearities (see below) has the regularity $\u\in\mathrm{L}^{\infty}(\epsilon,T;\D(\A))$, for any $\epsilon>0$ (\cite[Theorem 4.2]{KT2}, see \eqref{382} below). However, in the case of $\mathbb{T}^d$, we obtained the backward uniqueness result without using this regularity result.

To the best our knowledge, the backward uniqueness results for the system \eqref{1} and its stochastic counterpart are not considered in the literature. The main objectives of this work are to establish the following: 
\begin{enumerate}
	\item [(1)] Backward uniqueness property of the system \eqref{1} by using a logarithmic convexity method and its applications like
	\begin{itemize}
		\item [(a)] approximate controllability with respect to the initial data of the system \eqref{1},
		\item [(b)] zero Lipschitz deviation of the global attractor for the system \eqref{1}.
	\end{itemize}
	\item [(2)] The uniqueness of Lagrangian trajectories in 2D and 3D CBF flows by using the log-Lipschitz regularity.
\item [(3)] Pathwise backward uniqueness property of stochastic CBF equations perturbed by linear  multiplicative Gaussian noise and its application in  approximate controllability. 
\end{enumerate}

We use a logarithmic convexity approach to obtain the backward uniqueness results (cf. \cite{VBMR,JMG,IK}).  In order to prove the backward uniqueness, for $\u=\u_1-\u_2$, several authors have used the ratio  (cf. \cite{VBMR,ZBMN,JMG,IK,JCR5}) $$\Lambda(t)=\frac{\langle\mathcal{A}(t)\u(t),\u(t)\rangle}{\|\u(t)\|_{\H}^2},$$ for all $t\in[0,T]$,  where $\mathcal{A}(\cdot)$ is linear and self-adjoint. But in this work, for the deterministic supercritical case, we consider for all $t\in[0,T]$
\begin{align}\label{13}\Lambda(t)=\frac{\langle\mathcal{A}(\u(t)),\u(t)\rangle}{\|\u(t)\|_{\H}^2},\ \text{ where }\ \mathcal{A}(\u)=\mu\A\u+\alpha\u+\beta[\mathcal{C}(\u_1)-\mathcal{C}(\u_2)],\end{align} which is nonlinear. In the deterministic setting, the monotonicity property of $\mathcal{C}(\cdot)$ (see \eqref{2.23}  below) as well as the regularity of the solutions to the system \eqref{1} (see \eqref{05}, \eqref{18}, \eqref{010}, \eqref{327}, \eqref{011} and \eqref{334}  below) play a crucial role in obtaining the backward uniqueness results for the case $d=2,r\in[1,\infty)$, $d=3,r\in[3,\infty)$ ($2\beta\mu\geq 1$ for $d=r=3$). 

Whereas in the stochastic case, we use the ratio
\begin{align*}
	\widehat{\Lambda}(t)=\frac{\langle\widehat{\mathcal{A}}\v(t),\v(t)\rangle}{\|\v(t)\|_{\H}^2},  \text{ where }\ \widehat{\mathcal{A}}\v:=\mu\A\v+\left(\alpha+\frac{\sigma^2}{2}\right)\v,
	\end{align*}
for the noise intensity $\sigma\in\mathbb{R}\backslash\{0\}$. Note that $\widehat{\mathcal{A}}$ is a self-adjoint operator. By using this ratio, we are able to obtain the backward uniqueness results for the case $d=2,r\in[1,\infty)$ and $d=3,r\in[3,5]$ ($2\beta\mu\geq 1$ for $d=r=3$) only. The regularity estimates  obtained in \eqref{55} and \eqref{56} help us to obtain the desired results. As it is impossible to emulate estimates similar to \eqref{327}, \eqref{011}, and \eqref{334} on the time derivative in the stochastic case, we are unable to use the ratio given in \eqref{13} and consequently obtain the desired results for $d=3$, $r\in(5,\infty)$.

Approximate controllability is a property of dynamical systems which means that the system can be steered from any initial state to an arbitrary but close enough final state by inputting an appropriate control. It is a weaker form of exact controllability, which requires that the system can be steered from any initial state to an arbitrary final state. A direct consequence of the backward uniqueness result is the \emph{approximate controllability with respect to the initial data,  which is viewed as a start controller} (cf. \cite{VBMR,JLL}, etc.).  Following the works \cite{VBMR,CFJP}, we  prove the approximate controllability results for 2D and 3D deterministic as well as  stochastic  CBF equations.

The backward uniqueness of CBF equations has an additional application in the realm of attractor theory. The injectivity of the solution  semigroup $\S(\cdot)$ is an immediate consequence of backward uniqueness property (Lemma \ref{lem3.7}).   From \cite[Theorem 3.5]{KKMTM} (see \cite{KKMTM_RMP} for 2D CBF flows), we know that if the autonomous forcing $\f\in\H$, then the system \eqref{555}  possesses a global attractor $\mathscr{A}$ in $\H$. In fact, for the case $d=2,r\in[1,\infty)$ and $d=3,r\in[3,\infty)$ ($2\beta\mu\geq 1$ for $d=r=3$), using the procedure followed to prove  backward uniqueness, one can establish the following (Theorem  \ref{thm3.5}):
\begin{align}\label{1.5}
	\|\A^{\frac{1}{2}}(\u_1-\u_2)\|_{\H}^2\leq C_0\|\u_1-\u_2\|_{\H}^2\log\left(\frac{M_0^2}{\|\u_1-\u_2\|_{\H}^2}\right),\ \text{ for all }\ \u_1,\u_2\in\mathscr{A},\ \u_1\neq\u_2,
\end{align}
where $M_0\geq 4\sup\limits_{\x\in\H}\|\x\|_{\H}$ and $C_0$ is a constant. Note that the above result can be used to obtain the $1$-log-Lipschitz continuity of $\A : \mathscr{A}\to\H$ (Corollary \ref{cor3.9}).    Using the techniques given in \cite{EPJC,JCR5}, the estimate \eqref{1.5} is used to  prove the zero Lipschitz deviation for 2D and 3D CBF equations  when $\f\in\H$ (Theorem  \ref{thm3.6}). Whether \eqref{1.5} holds without the factor $\log\left(\frac{M_0^2}{\|\u_1-\u_2\|_{\H}^2}\right)$ is an open problem.

The authors in \cite{MDJC,MDJC1}  proved the uniqueness of Lagrangian trajectories for   2D and 3D  NSE (local for 3D case)  in periodic domains  and  bounded domains, respectively. The log-Lipschitz regularity of the velocity field  and H\"older regularity of the flow map of the three-dimensional Navier-Stokes equations  with small data in critical spaces is demonstrated in \cite{HBMC}. Given a smooth solution of 2D and 3D NSE, the authors in \cite{MDJC1} established  the uniqueness of Lagrangian particle trajectories, as well as their continuity with respect to the Eulerian initial data through the abstract results in \cite[Theorems 3.2.1 and 3.2.2]{MDJC1}.  By applying these abstract results, we  prove the uniqueness of Lagrangian trajectories in 2D and 3D CBF flows  and the continuity of Lagrangian trajectories with respect to the Eulerian initial data (Theorem \ref{thm3.11}).

\subsection{Organization of the paper}
The rest of the paper is organized as follows: In the next section, we provide the necessary function spaces and operators needed to obtain main results of this work. The backward uniqueness for 2D and 3D deterministic CBF equations for  $d=2,r\in[1,\infty)$ and $d=3,r\in[3,\infty)$ ($2\beta\mu\geq 1$ for $d=r=3$) is established in Section \ref{Sec3} by using a logarithmic convexity approach (Theorem \ref{thm3.1}). Some applications of the backward uniquness result are discussed in Section \ref{Sec4}. As a  direct consequence of the backward uniqueness result, we establish the approximate controllability with respect to the initial data in Theorem \ref{thm3.3}. As an application in the attractor theory, we show that the semigroup operator $\S(\cdot)$ is injective (Lemma \ref{lem3.7}). Moreover, we establish some results  on boundedness of ``log-Dirichlet quotients" for differences of solutions of the 2D and 3D CBF equations on the global attractor (Theorem \ref{thm3.5} and Corollary \ref{cor3.9}).  This  helps us to prove the zero Lipschitz deviation for the global attractors  (Theorem \ref{thm3.6}). The uniqueness as well as continuity with respect to the Eulerian initial data  of Lagrangian trajectories in 2D and 3D CBF flows is established in Theorem \ref{thm3.11}.   The stochastic CBF equations perturbed by a linear multiplicative Gaussian noise  is considered in Section \ref{Sec5}. By using a suitable transformation, we transform the stochastic CBF equations  into a random dynamical system and then prove the pathwise  backward uniqueness results for stochastic CBF equations for  $d=2,r\in[1,\infty)$ and $d=3,r\in[3,5]$ ($2\beta\mu\geq 1$ for $d=r=3$) in Theorem \ref{thm4.1}.  The approximate controllability result for the stochastic CBF equations is stated in Theorem \ref{thm4.2}. 

	\section{Functional Setting}\label{Sec2}\setcounter{equation}{0}
	In this section, we provide the necessary function spaces needed to obtain the main results of this work. We consider the problem \eqref{1} on a $d$-dimensional torus $\mathbb{T}^{d}=\big(\R/\mathrm{L}\mathbb{Z}\big)^{d}$ $(d=2,3)$, with periodic boundary conditions.

	\subsection{Function spaces} 
	Let $\C_{\mathrm{p}}^{\infty}(\mathbb{T}^d;\R^d)$ denote the space of all infinitely differentiable  functions $\u$ satisfying periodic boundary conditions $\u(x+\mathrm{L}e_{i},\cdot) = \u(x,\cdot)$, for $x\in \R^d$. \emph{We are not assuming the zero mean condition for the velocity field unlike the case of NSE, since the absorption term $\beta|\u|^{r-1}\u$ does not preserve this property (see \cite{MTT}). Therefore, we cannot use the well-known Poincar\'e inequality and we have to deal with the  full $\H^1$-norm.} The Sobolev space  $\H_{\mathrm{p}}^s(\mathbb{T}^d):=\mathrm{H}_{\mathrm{p}}^s(\mathbb{T}^d;\mathbb{R}^d)$ is the completion of $\C_{\mathrm{p}}^{\infty}(\mathbb{T}^d;\R^d)$  with respect to the $\H^s$-norm and the norm on the space $\H_{\mathrm{p}}^s(\mathbb{T}^d)$ is given by $$\|\u\|_{{\H}^s_{\mathrm{p}}}:=\left(\sum_{0\leq|\boldsymbol\alpha|\leq s}\|\D^{\boldsymbol\alpha}\u\|_{\mathbb{L}^2(\mathbb{T}^d)}^2\right)^{1/2}.$$ 	
	It is known from \cite{JCR_2001} that the Sobolev space of periodic functions $\H_{\mathrm{p}}^s(\mathbb{T}^d)$, for $s\geq0$, can be defined as 
	$$\H_{\mathrm{f}}^s(\mathbb{T}^d)=\left\{\u:\u=\sum_{k\in\mathbb{Z}^d}\u_{k}\mathrm{e}^{2\pi i k\cdot x /  \mathrm{L}},\ \overline{\u}_{k}=\u_{-k}, \  \|\u\|_{{\H}^s_\mathrm{f}}:=\left(\sum_{k\in\mathbb{Z}^d}(1+|k|^{2s})|\u_{k}|^2\right)^{1/2}<\infty\right\}.$$ We infer from \cite[Proposition 5.38]{JCR_2001} that the norms $\|\cdot\|_{{\H}^s_{\mathrm{p}}}$ and $\|\cdot\|_{{\H}^s_f}$ are equivalent. Let us define 
	\begin{align*} 
		\mathcal{V}:=\{\u\in\C_{\mathrm{p}}^{\infty}(\mathbb{T}^d;\R^d):\nabla\cdot\u=0\}.
	\end{align*}
	We define the spaces $\H$ and $\widetilde{\L}^{p}$ as the closure of $\mathcal{V}$ in the Lebesgue spaces $\mathrm{L}^2(\mathbb{T}^d;\R^d)$ and $\mathrm{L}^p(\mathbb{T}^d;\R^d)$ for $p\in(2,\infty)$, respectively. We also define the space $\V$ as the closure of $\mathcal{V}$ in the Sobolev space $\mathrm{H}^1(\mathbb{T}^d;\R^d)$. Then, we characterize the spaces $\H$, $\widetilde{\L}^p$ and $\V$ with the norms  $$\|\u\|_{\H}^2:=\int_{\mathbb{T}^d}|\u(x)|^2\d x,\quad \|\u\|_{\widetilde{\L}^p}^p:=\int_{\mathbb{T}^d}|\u(x)|^p\d x\ \text{ and }\ \|\u\|_{\V}^2:=\int_{\mathbb{T}^d}(|\u(x)|^2+|\nabla\u(x)|^2)\d x,$$ respectively. 
	Let $(\cdot,\cdot)$ denote the inner product in the Hilbert space $\H$ and $\langle \cdot,\cdot\rangle $ represent the induced duality between the spaces $\V$  and its dual $\V'$ as well as $\widetilde{\L}^p$ and its dual $\widetilde{\L}^{p'}$, where $\frac{1}{p}+\frac{1}{p'}=1$. Note that $\H$ can be identified with its own dual $\H'$. From \cite[Subsection 2.1]{FKS}, we have that the sum space $\V'+\widetilde{\L}^{p'}$ is well defined and  is a Banach space with respect to the norm 
	\begin{align*}
		\|\u\|_{\V'+\widetilde{\L}^{p'}}&:=\inf\{\|\u_1\|_{\V'}+\|\u_2\|_{\wi\L^{p'}}:\u=\u_1+\u_2, \u_1\in\V' \ \text{and} \ \u_2\in\wi\L^{p'}\}\nonumber\\&=
		\sup\left\{\frac{|\langle\u_1+\u_2,\f\rangle|}{\|\f\|_{\V\cap\widetilde{\L}^p}}:\boldsymbol{0}\neq\f\in\V\cap\widetilde{\L}^p\right\},
	\end{align*}
	where $\|\cdot\|_{\V\cap\widetilde{\L}^p}:=\max\{\|\cdot\|_{\V}, \|\cdot\|_{\wi\L^p}\}$ is a norm on the Banach space $\V\cap\widetilde{\L}^p$. Also the norm $\max\{\|\u\|_{\V}, \|\u\|_{\wi\L^p}\}$ is equivalent to the norms  $\|\u\|_{\V}+\|\u\|_{\widetilde{\L}^{p}}$ and $\sqrt{\|\u\|_{\V}^2+\|\u\|_{\widetilde{\L}^{p}}^2}$ on the space $\V\cap\widetilde{\L}^p$. Furthermore, we have
	$
	(\V'+\widetilde{\L}^{p'})'=	\V\cap\widetilde{\L}^p \  \text{and} \ (\V\cap\widetilde{\L}^p)'=\V'+\widetilde{\L}^{p'}.
	$
	Moreover, we have the continuous embeddings $\V\cap\widetilde{\L}^p\hookrightarrow\V\hookrightarrow\H\cong\H'\hookrightarrow\V'\hookrightarrow\V'+\widetilde{\L}^{p'},$ where the embedding $\V\hookrightarrow\H$ is compact. 
	
	\subsection{Linear operator}\label{linope}
	Let $\mathcal{P}_p: \L^p(\mathbb{T}^d) \to\wi\L^p,$ $p\in[1,\infty)$ be the Helmholtz-Hodge (or Leray) projection operator (cf.  \cite{JBPCK,DFHM}).	Note that $\mathcal{P}_p$ is a bounded linear operator and for $p=2$,  $\mathcal{P}:=\mathcal{P}_2$ is an orthogonal projection (see \cite[Section 2.1]{JCR4}). We define the Stokes operator 
	\begin{equation*}
		\left\{
		\begin{aligned}
			\A\u&:=-\mathcal{P}\Delta\u=-\Delta\u,\;\u\in\D(\A),\\
			\D(\A)&:=\V\cap{\H}^{2}_\mathrm{p}(\mathbb{T}^d). 
		\end{aligned}
		\right.
	\end{equation*}
	For the Fourier expansion $\u=\sum\limits_{k\in\mathbb{Z}^d} e^{2\pi i k\cdot x} \u_{k} ,$ we calculate by using Parseval's identity
	\begin{align*}
		\|\u\|_{\H}^2=\sum\limits_{k\in\mathbb{Z}^d} |\u_{k}|^2 \  \text{and} \ \|\A\u\|_{\H}^2=(2\pi)^4\sum_{k\in\mathbb{Z}^d}|k|^{4}|\u_{k}|^2.
	\end{align*}
	Therefore, we have 
	\begin{align*}
		\|\u\|_{\H^2_\mathrm{p}}^2=\sum_{k\in\mathbb{Z}^d}|\u_{k}|^2+\sum_{k\in\mathbb{Z}^d}|k|^{4}| \u_{k}|^2=\|\u\|_{\H}^2+\frac{1}{(2\pi)^4}\|\A\u\|_{\H}^2\leq\|\u\|_{\H}^2+\|\A\u\|_{\H}^2.
	\end{align*}
	Moreover, by the definition of $\|\cdot\|_{\H^2_\mathrm{p}}$, we have $	\|\u\|_{\H^2_\mathrm{p}}^2\geq\|\u\|_{\H}^2+\|\A\u\|_{\H}^2$ and hence it is immediate that both the norms are equivalent and  $\D(\I+\A)=\H^2_\mathrm{p}(\mathbb{T}^d)$. 
	
	Let us define a new operator $\mathscr{A}:\D(\mathscr{A})\subset \H\to\H$ by $\mathscr{A}:=\I+\A,$ where $\D(\mathscr{A})=\D(\A)$ and $\I$ is the identity operator on $\H$. It is clear that $\mathscr{A}$ is a non-negative self-adjoint and unbounded linear operator in $\H$. Then, we calculate
	\begin{align*}
		\|\mathscr{A}^{\frac{1}{2}}\u\|_{\H}^2=(\mathscr{A}^{\frac{1}{2}}\u,\mathscr{A}^{\frac{1}{2}}\u)=(\u,\mathscr{A}\u)=
		\|\u\|_{\H}^2+\|\nabla\u\|_{\H}^2=\|\u\|_{\V}^2,
	\end{align*}
	which yields that $\D(\mathscr{A}^{\frac{1}{2}})=\V$. Moreover, the inverse $(\I+\A)^{-1}$ is also self-adjoint. The invertibility of $\I+\A$  follows from the fact that $\A=-\Delta$ is an $m$-accretive operator in $\H$ (see \cite[Chapter 1]{VB1}), which says that $\mathrm{R}(\I+\A)=\H$. This means for every $\f\in\H$, $\u\in\D(\A)$, the equation $\A\u+\u=\f$ 
	has a unique solution. Also, we calculate
	\begin{align*}
		\|\u\|_{\V}^2=\|\u\|_{\H}^2+\|\nabla\u\|_{\H}^2=\|\u\|_{\H}^2+\langle\A\u,\u\rangle=(\f,\u)\leq\|\f\|_{\H}\|\u\|_{\H}\leq\|\f\|_{\H}\|\u\|_{\V},
	\end{align*}
	which implies that $\|\u\|_{\V}\leq\|\f\|_{\H}$, that is, $\|(\I+\A)^{-1}\f\|_{\V}\leq\|\f\|_{\H}.$ 
	Thus, $(\I+\A)^{-1}$ maps bounded subset of $\H$ into bounded subsets of $\V$. Since the embedding $\V\hookrightarrow\H$ is compact, therefore $(\I+\A)^{-1}$ is a compact operator on $\H$ and thus by the spectral mapping theorem, the spectrum of $\I+\A$ consists of an infinite sequence of eigenvalues $0<\lambda_1\leq\lambda_2\leq\ldots\lambda_k\leq\ldots,$ with $\lambda_k\to\infty$ as $k\to\infty$ such that  
	\begin{align*}
		\mathscr{A}\w_k=\lambda_k\w_k, \  \ \text{for} \ k=1,2,\ldots,
	\end{align*} 
	where $\w_k$'s are the corresponding eigenfunctions in $\D(\mathscr{A})$ which forms an orthonormal basis for $\H$. We can express $\u$ in the Fourier expansion form as $\u=\sum\limits_{k\in\mathbb{Z}^d}\u_{k} \mathrm{e}^{2\pi i k\cdot x /  \mathrm{L}}$ or we can write
	\begin{align*}
		\u=\sum\limits_{k\in\mathbb{Z}^d\setminus\{\boldsymbol{0}\}}\u_{k} \mathrm{e}^{2\pi i k\cdot\x /  \mathrm{L}} +\u_0 :=\wi{\u}+\u_0,
	\end{align*}
	where $\wi{\u}:=\sum\limits_{k\in\mathbb{Z}^d\setminus\{\boldsymbol{0}\}}\u_{k} \mathrm{e}^{2\pi i k\cdot x /  \mathrm{L}}$ and $\u_0$ is the element corresponding to $k=(0,\ldots,0)$. 
	Then from \cite[Chapter 6]{JCR4}, we calculate 
	\begin{align*}
		\A\u=-\Delta\u=\left(\frac{4\pi^2}{\mathrm{L}^2}\right)\sum\limits_{k\in\mathbb{Z}^d}|k|^2\u_{k} \mathrm{e}^{2\pi i k\cdot x /  \mathrm{L}},
	\end{align*}
	and therefore $\mathscr{A}\u=(\I+\A)\u=\sum\limits_{k\in\mathbb{Z}^d}\left(1+\frac{4\pi^2}{\mathrm{L}^2}\right)|k|^2\u_{k} \mathrm{e}^{2\pi i k\cdot x /  \mathrm{L}}.$
	It shows that $\u_k$ is an eigenfunction of $\mathscr{A}$ corresponding to the eigenvalue $1+\tilde{\lambda}_k$, where $\tilde{\lambda}_k=\frac{4\pi^2}{\mathrm{L}^2}|k|^2$. In particular, the operator $\mathscr{A}$ and the Stokes operator $\A$ have the same eigenfunctions corresponding to eigenvalues $1+\tilde{\lambda}_k$ and $\tilde{\lambda}_k$, respectively. For $k=\boldsymbol{0}$ and any constant vector $\w_0,$   we have $\mathscr{A}\w_0=(\I+\A)\w_0=\w_0=(1+0)\w_0$. Thus, $\w_0$ is an eigenfunction of $\mathscr{A}$ corresponding to the eigenvalue $\lambda_0=1$. Note that $0$ is an eigenvalue of the  Stokes operator $\A$.
	
	From \cite[Chapter 6]{JCR4}, the eigenfunctions of the operator $\mathscr{A}=\I+\A$ on the torus $\mathbb{T}^d$ are given by  
	\begin{align}\label{ef}
		\w_k^{(s)}:=\sqrt{\frac{2}{\mathrm{L}^d}}\sin\left(2\pi k\cdot\frac{x}{\mathrm{L}}\right),  \ \w_k^{(c)}:=\sqrt{\frac{2}{\mathrm{L}^d}} \cos\left(2\pi k\cdot\frac{x}{\mathrm{L}}\right) \ \ \text{and} \ \w_k^{(0)}=\w_0,
	\end{align}
	with eigenvalues $\lambda_k=1+\frac{4\pi^2}{\mathrm{L}^2}|k|^2$. We can order these eigenvalues in a non-decreasing manner such that for every $\lambda_k $, $k\in\mathbb{Z}^d\setminus\{\boldsymbol{0}\}$, one can find a corresponding eigenvalue $\lambda_m$ for some appropriate $m\in\N$, with $\lambda_{m+1}\geq\lambda_{m}$ and the corresponding eigenfunction $\w_k=\w_m$ (cf. \cite[pp. 52]{FMRT}).
 For $\u\in \H$ and  $\alpha>0,$ one can define
$\A^\alpha \u=\sum_{k=1}^\infty \lambda_k^\alpha \u_k \boldsymbol{e}_k,  \ \u\in\D(\A^\alpha), $ where $\D(\A^\alpha)=\left\{\u\in \H:\sum_{k=1}^\infty \lambda_k^{2\alpha}|\u_k|^2<+\infty\right\},$ 
and   $\D(\A^\alpha)$ is equipped with the norm 
$
	\|\A^\alpha \u\|_{\H}=\left(\sum_{k=1}^\infty \tilde{\lambda}_k^{2\alpha}|\u_k|^2\right)^{1/2}.
$

	\subsection{Bilinear operator}
	Let us define the \textsl{trilinear form} $b(\cdot,\cdot,\cdot):\V\times\V\times\V\to\R$ by $$b(\u,\v,\w)=\int_{\mathbb{T}^d}(\u(x)\cdot\nabla)\v(x)\cdot\w(x)\d x=\sum_{i,j=1}^d\int_{\mathbb{T}^d}\u_i(x)\frac{\partial \v_j(x)}{\partial x_i}\w_j(x)\d x.$$ If $\u, \v$ are such that the linear map $b(\u, \v, \cdot) $ is continuous on $\V$, the corresponding element of $\V'$ is denoted by $\mathcal{B}(\u, \v)$. We also denote $\mathcal{B}(\u) = \mathcal{B}(\u, \u)=\mathcal{P}[(\u\cdot\nabla)\u]$.
	An integration by parts yields 
	\begin{equation}\label{b0}
		\left\{
		\begin{aligned}
			b(\u,\v,\w) &=  -b(\u,\w,\v),\ \text{ for all }\ \u,\v,\w\in \V,\\
			b(\u,\v,\v) &= 0,\ \text{ for all }\ \u,\v \in\V.
		\end{aligned}
		\right.
	\end{equation}

	The following estimates on the trilinear form $b(\cdot,\cdot,\cdot)$ is useful in the sequel (\cite[Chapter 2, Section 2.3]{Te1}):
		\begin{itemize}
			\item [$(i)$]
			For $d=2$, 
			\begin{align}\label{b2}
				|b(\u,\v,\w)|&\leq C\begin{cases}
					\|\u\|^{1/2}_{\H}\|\u\|^{1/2}_{\V}\|\v\|_{\V}\|\w\|^{1/2}_{\H}\|\w\|^{1/2}_{\V}, &\text{for all }\  \u, \v, \w\in \V,\\
					\|\u\|^{1/2}_{\H}\|\u\|^{1/2}_{\V}\|\v\|^{1/2}_{\V}\|\v\|^{1/2}_{\H^2}\|\w\|_{\H}, &\text{for all }\ \u\in \V, \v\in \D(\A), \w\in \H,\\
					\|\u\|_{\H}^{1/2}\|\u\|_{\H^2}^{1/2}\|\v\|_{\V}\|\w\|_{\H},&\text{for all }\ \u\in\D(\A),\v\in\V,\w\in\H. 
				\end{cases}
			\end{align}
			
			\item[$(ii)$]
			For $d=3$,
			\begin{align}\label{b4}
				|b(\u,\v,\w)|&\leq C\begin{cases}
					\|\u\|^{1/4}_{\H}\|\u\|^{3/4}_{\V}\|\v\|_{\V}\|\w\|^{1/4}_{\H}\|\w\|^{3/4}_{\V}, &\text{for all }\  \u, \v, \w\in \V,\\
					\|\u\|_{\V}\|\v\|_{\V}^{1/2}\|\v\|_{\H^2}^{1/2}\|\w\|_{\H}, &\text{for all }\ \u\in\V, \v\in\D(\A), \w\in\H, \\
					\|\u\|_{\V}^{1/2}\|\u\|_{\H^2}^{1/2}\|\v\|_{\V}\|\w\|_{\H},&\text{for all }\ \u\in\D(\A),\v\in\V,\w\in\H. 
				\end{cases}
			\end{align}
		\end{itemize}
	The above estimates can be derived by using H\"oder's, Sobolev's, Ladyzhenskaya's, Agmon's and Gagliardo-Nirenberg's inequalities. 

	\subsection{Nonlinear operator}
	Let us now consider the operator $\mathcal{C}(\u):=\mathcal{P}(|\u|^{r-1}\u)$. It is immediate that $\langle\mathcal{C}(\u),\u\rangle =\|\u\|_{\widetilde{\L}^{r+1}}^{r+1}$. From  \cite[Subsection 2.4]{MT2}, we have 
	\begin{align}\label{2.23}
		\langle\mathcal{C}(\u)-\mathcal{C}(\v),\u-\v\rangle&\geq \frac{1}{2}\||\u|^{\frac{r-1}{2}}(\u-\v)\|_{\H}^2+\frac{1}{2}\||\v|^{\frac{r-1}{2}}(\u-\v)\|_{\H}^2\nonumber\\&\geq \frac{1}{2^{r-1}}\|\u-\v\|_{\wi\L^{r+1}}^{r+1}\geq 0,
	\end{align}
	for $r\geq 1$.  The map $\mathcal{C}(\cdot):\widetilde{\L}^{r+1}\to\widetilde{\L}^{\frac{r+1}{r}}$ is Gateaux differentiable with Gateaux derivative 
	\begin{align}\label{2p9}
		\mathcal{C}'(\u)\v&=\left\{\begin{array}{cl}\mathcal{P}(\v),&\text{ for }r=1,\\ \left\{\begin{array}{cc}\mathcal{P}(|\u|^{r-1}\v)+(r-1)\mathcal{P}\left(\frac{\u}{|\u|^{3-r}}(\u\cdot\v)\right),&\text{ if }\u\neq \boldsymbol{0},\\\boldsymbol{0},&\text{ if }\u=\boldsymbol{0},\end{array}\right.&\text{ for } 1<r<3,\\ \mathcal{P}(|\u|^{r-1}\v)+(r-1)\mathcal{P}(\u|\u|^{r-3}(\u\cdot\v)), &\text{ for }r\geq 3,\end{array}\right.
	\end{align}
	for all $\u,\v\in\widetilde{\L}^{r+1}$. Moreover, for $r\geq3$, $\mathcal{C}(\cdot)$ is twice Gateaux differentiable  with second order Gateaux derivative
	\begin{align}\label{30}
	\mathcal{C}''(\u)(\v\otimes\w)=\left\{\begin{array}{ll} (r-1)\mathcal{P}\left\{|\u|^{r-3}\left[(\u\cdot\w)\v+(\u\cdot\v)\w+(\w\cdot\v)\u\right]\right\}& \\+\left\{\begin{array}{cl}(r-1)(r-3)\mathcal{P}\left[\frac{\u}{|\u|^{5-r}}(\u\cdot\v)(\u\cdot\w)\right],&\text{ for }\u\neq \boldsymbol{0},\\ \boldsymbol{0}&\text{ for }\u=\boldsymbol{0},\end{array}\right.&\text{ for }3\leq r< 5,\\(r-1)\mathcal{P}\left\{|\u|^{r-3}\left[(\u\cdot\w)\v+(\u\cdot\v)\w+(\w\cdot\v)\u\right]\right\}& \\+(r-1)(r-3)\mathcal{P}\left[|\u|^{r-5}(\u\cdot\v)(\u\cdot\w)\u\right],&\text{ for }r\geq 5,\end{array}\right.
\end{align}
	for all $\u,\v,\w\in\widetilde{\L}^{r+1}$.
	\begin{remark}\label{rem2.1}
		On a torus (cf. \cite{KWH,MT2}), we have 
		\begin{align}\label{P.11}
			\|\u\|_{\widetilde{\L}^{3(r+1)}}^{r+1}\leq C\int_{\mathbb{T}^d}|\nabla \u(x)|^2|\u(x)|^{r-1}\d x + C\int_{\mathbb{T}^d}|\u(x)|^{r+1}\d x,
		\end{align}
		for $d=3$ and $r\geq 1$. Moreover, we obtain 
		\begin{align}\label{P.22}
			\|\u\|_{\widetilde{\L}^{p(r+1)}}^{r+1}  =\||\u|^{\frac{r+1}{2}}\|_{\wi\L^{2p}}^{2}
			  & \leq C \int_ {\mathbb{T}^d}|\nabla|\u|^{\frac{r+1}{2}}|^{2}\d x + C\int_{\mathbb{T}^d}|\u(x)|^{r+1}\d x 
			\nonumber\\ & \leq C\int_{\mathbb{T}^d}|\nabla \u(x)|^2|\u(x)|^{r-1}\d x + C\int_{\mathbb{T}^d}|\u(x)|^{r+1}\d x,
		\end{align}
		for $d=2$ and for all $p\in[2,\infty).$	
	\end{remark}

\begin{remark}
\textbf{1.}	We obtain from Agmon's inequality (for $d=2$) that
	\begin{align}\label{Agmon2}
		\hspace{-8mm}	\|\u\|_{\wi\L^{\infty}}\leq \|\u\|^{\frac{1}{2}}_{\H}\|\u\|^{\frac{1}{2}}_{\H^2}= \|\u\|^{\frac{1}{2}}_{\H}\|(\mathrm{I}+\A)\u\|^{\frac{1}{2}}_{\H} \leq C \left[\|\u\|_{\H}+\|\u\|^{\frac{1}{2}}_{\H}\|\A\u\|^{\frac{1}{2}}_{\H}\right].
	\end{align}	
\textbf{2.}	 Moreover, Agmon's inequality (for $d=3$) yield 
	\begin{align}\label{Agmon3}
		\|\u\|_{\wi\L^{\infty}} & \leq \|\u\|^{\frac{1}{2}}_{\H^1}\|\u\|^{\frac{1}{2}}_{\H^2}= \|\u\|^{\frac{1}{2}}_{\V}\|(\mathrm{I}+\A)\u\|^{\frac{1}{2}}_{\H} 
		\nonumber\\ & \leq C \left[\|\u\|^{\frac{1}{2}}_{\H}\|\u\|^{\frac{1}{2}}_{\V}+\|\u\|^{\frac{1}{2}}_{\V}\|\A\u\|^{\frac{1}{2}}_{\H}\right]
		\nonumber\\ & \leq C \left[\|\u\|_{\H} + \|\u\|^{\frac{1}{2}}_{\H}\|\nabla\u\|^{\frac{1}{2}}_{\H}+\|\u\|^{\frac{1}{2}}_{\H}\|\A\u\|^{\frac{1}{2}}_{\H}++\|\nabla\u\|^{\frac{1}{2}}_{\H}\|\A\u\|^{\frac{1}{2}}_{\H}\right].\hspace{4mm}
	\end{align}	
\end{remark}

Taking the Helmholtz-Hodge projection onto the system \eqref{1}, we rewrite 
\begin{equation}\label{555}
	\left\{
	\begin{aligned}
		\frac{\d  \u(t)}{\d  t}+\mu \A\u(t)+\B(\u(t))+\alpha\u(t)+\beta\mathcal{C}(\u(t))&=\boldsymbol{f}, \\
		\u(0)&=\boldsymbol{x},
	\end{aligned}\right. 
\end{equation} 
where for simplicity of notation we used $\mathcal{P}\f$ as $\f$. For the case $d=2,\ r\in[1,\infty)$ and $d=3,\ r\in[3,\infty)$ ($2\beta\mu\geq 1$ for $d=r=3$),  for $\x\in\H$ and $\f\in\mathrm{L}^2(0,T;\V')$, the existence and uniqueness of global Leray-Hopf weak solutions 
$$\u\in\mathrm{C}([0,T];\H)\cap\mathrm{L}^2(0,T;\V)\cap\mathrm{L}^{r+1}(0,T;\wi\L^{r+1})$$ 
satisfying the energy equality 
\begin{align}\label{211}
	&\|\u(t)\|_{\H}^2+2\mu\int_0^t\|\nabla\u(s)\|_{\H}^2\d s+2\alpha\int_0^t\|\u(s)\|_{\H}^2\d s+2\beta\int_0^t\|\u(s)\|_{\wi\L^{r+1}}^{r+1}\d s\nonumber\\&=\|\x\|_{\H}^2+2\int_0^t\langle\f(s),\u(s)\rangle\d s,
\end{align}
for all $t\in[0,T]$  of the system \eqref{555} is established in \cite{KWH,KT2,MT1}. Moreover, for  $\f\in\mathrm{L}^2(0,T;\H)$, one can obtain the existence of a unique strong solution $\u\in\C((0,T];\V)\cap \mathrm{L}^2(\epsilon,T;\D(\A))\cap\mathrm{L}^{r+1}(\epsilon,T;\wi\L^{3(r+1)}), $ for any $\epsilon>0$  to the system \eqref{1} (cf. \cite{KWH}). 
	\section{Backward Uniqueness and Applications}\label{Sec3}\setcounter{equation}{0}
	In this section, we prove the backward uniqueness result for the system \eqref{555} for the case $d=2,\ r\in[1,\infty)$ and $d=3,\ r\in[3,\infty)$ ($2\beta\mu\geq 1$ for $d=r=3$) by using a log convexity method. We use the regularity estimates of the system \eqref{555} to obtain the required results. We follow the works \cite{VBMR,JMG,IK,LEP,JCR5} to establish the backward uniqueness results. 
	\begin{theorem}[Backward uniqueness]\label{thm3.1}
		Let $\x\in\H$, $\f\in\mathrm{W}^{1,2}(0,T;\H)$ and  $\u_1,\u_2$ satisfy the first equation in the system \eqref{555}. If $\u_1(T)=\u_2(T)$ in $\H$, then $\u_1(t)=\u_2(t)$ in $\H$ for all $t\in[0,T]$. 
		\end{theorem}
	\begin{proof}
		The proof of this theorem has been divided into the following steps: For completeness, we  provide the forward uniqueness result also. 
		\vskip 0.2cm
	\noindent 	\textbf{Step 1:} \textsl{Energy equality and forward uniqueness:} 
		Let us first consider the case $d=2,3$ and $r\in(3,\infty)$.	
		Taking the inner product with $\u(\cdot)$ to the first equation in \eqref{555} and then integrating from $0$ to $T$, we obtain  for all $t\in[0,T]$
		\begin{align*}
			&\|\u(t)\|_{\H}^2+2\mu\int_0^t\|\nabla\u(s)\|_{\H}^2\d s+2\alpha\int_0^t\|\u(s)\|_{\H}^2\d s+2\beta\int_0^t\|\u(s)\|_{\wi\L^{r+1}}^{r+1}\d s\nonumber\\&=\|\x\|_{\H}^2+2\int_0^t(\f(s),\u(s)) \d s
			  \leq \|\x\|_{\H}^2+\alpha\int_0^t\|\u(s)\|_{\H}^2\d s+\frac{1}{\alpha}\int_0^t\|\f(s)\|_{\H}^2\d s. 
		\end{align*}
	Therefore, we have 
	\begin{align}\label{02}
		&\|\u(t)\|_{\H}^2+2\mu\int_0^t\|\nabla\u(s)\|_{\H}^2\d s + \alpha\int_0^t\|\u(s)\|_{\H}^2\d s+2\beta\int_0^t\|\u(s)\|_{\wi\L^{r+1}}^{r+1}\d s\nonumber\\&\leq \|\x\|_{\H}^2+\frac{1}{\alpha}\int_0^T\|\f(t)\|_{\H}^2\d t =: K,
	\end{align}
	for all $t\in[0,T]$. One can establish the forward uniqueness of the system \eqref{1} in the following way: Let $\u_1(\cdot)$ and $\u_2(\cdot)$ be two weak solutions of the system \eqref{555} with the same initial  data and forcing, say $\x$ and $\f$, respectively. Then $\u=\u_1-\u_2$ satisfies the following system in $\V'+\wi\L^{\frac{r+1}{r}}$ for a.e. $t\in[0,T]$
		\begin{align}
			\left\{
			\begin{aligned}
				\frac{\d  \u(t)}{\d  t}+\mu\A\u(t)+[\B(\u_1(t),\u(t))+\B(\u(t),\u_2(t))]+\alpha\u(t)+\beta[\mathcal{C}(\u_1(t))-\mathcal{C}(\u_1(t))]&=\boldsymbol{0},\\
				\u(0)&=\boldsymbol{0}.
			\end{aligned}\right. 
		\end{align}
	Taking the inner product with $\u(\cdot)$ and integrating from $0$ to $t$, we obtain 
	\begin{align}\label{04}
		&\|\u(t)\|_{\H}^2+2\mu\int_0^t\|\nabla\u(s)\|_{\H}^2\d s+2\alpha\int_0^t\|\u(s)\|_{\H}^2\d s+2\beta\int_0^t\langle\mathcal{C}(\u_1(s))-\mathcal{C}(\u_1(s)),\u(s)\rangle\d s\nonumber\\&=\|\u(0)\|_{\H}^2-2\int_0^t\langle\B(\u(s),\u_2(s)),\u(s)\rangle\d s,
		\end{align}
	for all $t\in[0,T]$. Note that $\langle\B(\u_1)-\B(\u_2),\u\rangle=\langle\B(\u,\u_2),\u\rangle,$ since $\langle\B(\u_1,\u),\u\rangle=0$. 
For $d=2,3$ and $r\in(3,\infty)$,  using H\"older's and Young's inequalities, we estimate $|\langle\B(\u,\u_2),\u\rangle|$ as  
\begin{align}\label{2p28}
	|\langle\B(\u,\u),\u_2\rangle|&\leq\|\nabla\u\|_{\H}\||\u_2|\u\|_{\H}\leq\frac{\mu }{2}\|\nabla\u\|_{\H}^2+\frac{1}{2\mu }\||\u_2|\u\|_{\H}^2.
\end{align}
We take the term $\||\u_2|\u\|_{\H}^2$ from \eqref{2p28} and use H\"older's and Young's inequalities to estimate it as 
\begin{align}\label{2.29}
	\int_{\mathcal{O}}|\u_2(x)|^2|\u(x)|^2\d x&=\int_{\mathcal{O}}|\u_2(x)|^2|\u(x)|^{\frac{4}{r-1}}|\u(x)|^{\frac{2(r-3)}{r-1}}\d x\nonumber\\&\leq\left(\int_{\mathcal{O}}|\u_2(x)|^{r-1}|\u(x)|^2\d x\right)^{\frac{2}{r-1}}\left(\int_{\mathcal{O}}|\u(x)|^2\d x\right)^{\frac{r-3}{r-1}}\nonumber\\&\leq\frac{\beta\mu }{2}\left(\int_{\mathcal{O}}|\u_2(x)|^{r-1}|\u(x)|^2\d x\right)+\frac{r-3}{r-1}\left(\frac{4}{\beta\mu (r-1)}\right)^{\frac{2}{r-3}}\left(\int_{\mathcal{O}}|\u(x)|^2\d x\right),
\end{align}
for $r>3$. Therefore, from \eqref{2p28}, we have 
	\begin{align}\label{347}
		&|\langle\B(\u,\u_2),\u\rangle|\leq\frac{\mu }{2}\|\nabla\u\|_{\H}^2+\frac{\beta}{4}\||\u_2|^{\frac{r-1}{2}}\u\|_{\H}^2+\frac{\vartheta}{2\mu}\|\u\|_{\H}^2,
	\end{align}
where $\vartheta=\frac{r-3}{r-1}\left(\frac{4}{\beta\mu (r-1)}\right)^{\frac{2}{r-3}}$. Moreover, from \eqref{2.23}, we have 
	\begin{align}\label{261}
	2	\beta\langle\mathcal{C}(\u_1)-\mathcal{C}(\u_2),\u\rangle & \geq \beta \||\u_1|^{\frac{r-1}{2}}\u\|_{\H}^2+ \beta \||\u_2|^{\frac{r-1}{2}}\u\|_{\H}^2 
		\nonumber\\ &  \geq \frac{\beta}{2}\||\u_1|^{\frac{r-1}{2}}\u\|_{\H}^2+\frac{\beta}{2}\||\u_2|^{\frac{r-1}{2}}\u\|_{\H}^2 + \frac{\beta}{2^{r-1}}\|\u\|_{\wi\L^{r+1}}^{r+1}
	\end{align}
Using \eqref{347} and \eqref{261} in \eqref{04}, we deduce
	\begin{align}\label{07}
	&\|\u(t)\|_{\H}^2+\mu\int_0^t\|\nabla\u(s)\|_{\H}^2\d s+2\alpha\int_0^t\|\u(s)\|_{\H}^2\d s+\frac{\beta}{2^{r-1}}\int_0^t\|\u(s)\|_{\wi\L^{r+1}}^{r+1}\d s\nonumber\\&\leq\|\u(0)\|_{\H}^2+\frac{\vartheta}{\mu}\int_0^t\|\u(s)\|_{\H}^2\d s,
\end{align}
 An application of Gronwall's inequality in \eqref{07} yields  for all $t\in[0,T]$
\begin{align}
	\|\u(t)\|_{\H}^2\leq\|\u(0)\|_{\H}^2e^{\frac{\vartheta T}{\mu} },
\end{align}
and the forward uniqueness follows.  For the case $d=r=3,$ we estimate $\langle\B(\u,\u_2),\u\rangle$ as 
\begin{align}\label{08}
	|\langle\B(\u,\u_2),\u\rangle|\leq\|\nabla\u\|_{\H}\||\u_2|\u\|_{\H}\leq\theta\mu\|\nabla\u\|^2_{\H}+\frac{1}{4\theta\mu}\||\u_2|\u\|_{\H}^2,
\end{align}
for some $0<\theta\leq 1$.
Using \eqref{261} and \eqref{08}, one can deduce  from \eqref{04} that 
\begin{align}
&	\|\u(t)\|_{\H}^2+(1-\theta)\mu\int_0^t\|\nabla\u(s)\|_{\H}^2\d s 
+\left(\beta-\frac{1}{2\theta\mu}\right)\int_0^t\||\u_2(s)|\u(s)\|_{\H}^2\d s\leq\|\u(0)\|_{\H}^2,
\end{align}
for all $t\in[0,T]$. For $2\theta\beta\mu\geq 1$, one can obtain the required result. For $d=2$ and $r\in[1,3]$, we estimate  $\langle\B(\u,\u_2),\u\rangle$ using H\"older's, Ldayzhenkaya's and Young's inequalities as 
\begin{align}\label{11}
	|\langle\B(\u,\u_2),\u\rangle|&\leq\|\u_2\|_{\wi\L^4}\|\nabla\u\|_{\H}\|\u\|_{\wi\L^4}\leq 2^{1/4}\|\u_2\|_{\wi\L^4}\|\nabla\u\|_{\H}^{3/2}\|\u\|_{\H}^{1/2}\nonumber\\&\leq\frac{\mu}{2}\|\nabla\u\|_{\H}^2+\frac{27}{16\mu^3}\|\u_2\|_{\wi\L^4}^4\|\u\|_{\H}^2. 
\end{align}
Combining \eqref{261} and \eqref{11}, and substituting it in \eqref{04}, we obtain 
\begin{align}\label{12}
	&\|\u(t)\|_{\H}^2+\mu\int_0^t\|\nabla\u(s)\|_{\H}^2\d s+2\alpha\int_0^t\|\u(s)\|_{\H}^2\d s+\frac{\beta}{2^{r-2}}\int_0^t\|\u(s)\|_{\wi\L^{r+1}}^{r+1}\d s\nonumber\\&\leq\|\u(0)\|_{\H}^2+\frac{27}{8\mu^3}\int_0^t\|\u_2(s)\|_{\wi\L^4}^4\|\u(s)\|_{\H}^2\d s,
\end{align}
for all $t\in[0,T]$. An application of Gronwall's inequality and then Ladyzhenskaya's inequality  in \eqref{12} yields  for all $t\in[0,T]$ 
\begin{align}
	\|\u(t)\|_{\H}^2\leq\|\u(0)\|_{\H}^2\exp\left(\frac{27}{4\mu^3}\sup_{t\in[0,T]}\|\u(t)\|_{\H}^2\int_0^T\|\u_2(t)\|_{\V}^2\d t\right),
\end{align}
and the uniqueness follows. 
		
		\vskip 0.2cm
	\noindent 	\textbf{Step 2:} \textsl{Further energy estimates:} For $d=2,3$ and $r\in(3,\infty)$, 	taking the inner product with $\A\u(\cdot)$  to the first equation in \eqref{555}  and then integrating from $\epsilon$ to $t$, we find 
		\begin{align}\label{03}
			&	\|\nabla\u(t)\|_{\H}^2+2\mu\int_{\epsilon}^t\|\A\u(s)\|_{\H}^2\d s+2\alpha\int_{\epsilon}^t\|\nabla\u(s)\|_{\H}^2\d s\nonumber\\&\quad+2\beta\int_{\epsilon}^t\||\u(s)|^{\frac{r-1}{2}}|\nabla\u(s)|\|_{\H}^2\d s+8\beta\left[\frac{(r-1)}{(r+1)^2}\right]\int_{\epsilon}^t\|\nabla|\u(s)|^{\frac{r+1}{2}}\|_{\H}^2\d s\nonumber\\&=\|\nabla\u(\epsilon)\|_{\H}^2+\left\{\begin{array}{cc}2\int_0^t(\f(s),\A\u(s))\d s,&\text{ for }d=2,\\
				2	\int_{\epsilon}^t(\B(\u(s)),\A\u(s))\d s+2\int_0^t(\f(s),\A\u(s))\d s,&\text{ for } d=3,\end{array}\right.\nonumber\\&\leq \|\nabla\u(\epsilon)\|_{\H}^2+\left\{\begin{array}{cc}\mu	\int_{\epsilon}^t\|\A\u(s)\|_{\H}^2\d s+\frac{1}{\mu}\int_0^t\|\f(s)\|_{\H}^2\d s,&\text{ for }d=2,\\
				\mu	\int_{\epsilon}^t\|\A\u(s)\|_{\H}^2\d s+\beta\int_{\epsilon}^t\||\u(s)|^{\frac{r-1}{2}}|\nabla\u(s)|\|_{\H}^2\d s&\\+\frac{2\vartheta}{\mu}\int_{0}^t\|\nabla\u(s)\|_{\H}^2\d s+\frac{2}{\mu}\int_0^t\|\f(s)\|_{\H}^2\d s,&\text{ for } d=3,\end{array}\right.
		\end{align}
		where we have used the fact that $(\B(\u),\A\u)=0$ in $d=2$ (\cite[Lemma 3.1]{Te1}) and  performed a calculation similar to \eqref{347}. Therefore, integrating the above inequality for $\epsilon\in(0,t)$ and then using \eqref{02}, we have $\text{for all } t\in[\epsilon,T],$
		\begin{align}\label{16}
			\|\nabla\u(t)\|_{\H}^2\leq \left\{\begin{array}{lc}\frac{K}{\mu  t}+\frac{1}{\mu}\int_0^T\|\f(t)\|_{\H}^2\d t,&\text{ for }d=2,\\
			\frac{K}{ \mu t}+\frac{2K\vartheta}{\mu^2}+\frac{2}{\mu}\int_0^T\|\f(t)\|_{\H}^2\d t,&\text{ for }d=3. \end{array}\right.
		\end{align}
		Therefore, from \eqref{03}, we further deduce 
		\begin{align}\label{05}
		&	\mu\int_{\epsilon}^t\|\A\u(s)\|_{\H}^2\d s+\beta\int_{\epsilon}^t\||\u(s)|^{\frac{r-1}{2}}|\nabla\u(s)|\|_{\H}^2\d s\leq \left\{\begin{array}{lc}\frac{K}{\mu \epsilon}+\frac{2}{\mu}\int_0^T\|\f(t)\|_{\H}^2\d t,&\text{ for }d=2,\\
			\frac{K}{\mu \epsilon}+\frac{4K\vartheta}{\mu^2}+\frac{4}{\mu}\int_0^T\|\f(t)\|_{\H}^2\d t,&\text{ for }d=3,\end{array}\right.
		\end{align}
		$\text{for all }\ t\in[\epsilon,T]$. Using Remark \ref{rem2.1}, 
	we obtain from \eqref{02} and \eqref{05} that
	\begin{align}\label{18}
		\int_{\epsilon}^t\|\u(s)\|_{\wi\L^{3(r+1)}}^{r+1}\d s\leq\left\{\begin{array}{lc}C\left(K+\frac{K}{\mu  \epsilon}+\frac{2}{\mu}\int_0^T\|\f(t)\|_{\H}^2\d t\right),&\text{ for }d=2,\\
		C\left(K +	\frac{K}{\mu \epsilon}+\frac{4K\vartheta}{\mu^2}+\frac{4}{\mu}\int_0^T\|\f(t)\|_{\H}^2\d t\right),&\text{ for }d=3.\end{array}\right.
	\end{align}
For $d=r=3$, we estimate $|(\B(\u),\A\u)|$ as 
\begin{align*}
	|(\B(\u),\A\u)|\leq\|\A\u\|_{\H}\||\u||\nabla\u|\|_{\H}\leq\frac{\theta\mu}{2}\|\A\u\|_{\H}^2+\frac{1}{2\theta\mu}\||\u||\nabla\u|\|_{\H}^2,
	\end{align*}
for some $0<\theta\leq 1$. Thus, from \eqref{03}, we deduce 
\begin{align}
		&\|\nabla\u(t)\|_{\H}^2+\mu(1-\theta)\int_{\epsilon}^t\|\A\u(s)\|_{\H}^2\d s+2\alpha\int_{\epsilon}^t\|\nabla\u(s)\|_{\H}^2\d s+2\left(\beta-\frac{1}{2\theta\mu}\right)\int_{\epsilon}^t\||\u(s)||\nabla\u(s)|\|_{\H}^2\d s\nonumber\\&\leq\|\nabla\u(\epsilon)\|_{\H}^2+\frac{1}{\mu}\int_0^t\|\f(s)\|_{\H}^2\d s,
\end{align}
for all $t\in[\epsilon,T]$ and some $0<\theta\leq 1$. Therefore, for $2\beta\mu\geq 1$, a calculation similar to \eqref{05} yields 
\begin{align}
&	\|\nabla\u(t)\|_{\H}^2+\mu(1-\theta)	\int_{\epsilon}^t\|\A\u(s)\|_{\H}^2\d s+2\left(\beta-\frac{1}{2\theta\mu}\right)\int_{\epsilon}^t\||\u(s)||\nabla\u(s)|\|_{\H}^2\d s
\nonumber\\&\leq\frac{K}{\mu  \epsilon}+\frac{1}{\mu}\int_0^T\|\f(t)\|_{\H}^2\d t,
\end{align}
for all $t\in[\epsilon,T]$. 

	Taking the inner product with $\frac{\d\u}{\d t}$ to the first equation in \eqref{555}, we obtain  for a.e. $t\in[\epsilon,T]$
\begin{align*}
	&	\left\|\frac{\d\u(t)}{\d t}\right\|_{\H}^2+\frac{\mu}{2}\frac{\d}{\d t}\|\nabla\u(t)\|_{\H}^2+\frac{\alpha}{2}\frac{\d}{\d t}\|\u(t)\|_{\H}^2+\frac{\beta}{r+1}\frac{\d}{\d t}\|\u(t)\|_{\wi\L^{r+1}}^{r+1}\nonumber\\&=-\left(\B(\u(t)),\frac{\d\u(t)}{\d t}\right)+\left(\f(t),\frac{\d\u(t)}{\d t}\right)\nonumber\\&\leq\left\|\frac{\d\u(t)}{\d t}\right\|_{\H}\||\u(t)||\nabla\u(t)|\|_{\H}+\|\f(t)\|_{\H}\left\|\frac{\d\u(t)}{\d t}\right\|_{\H}\nonumber\nonumber\\&\leq\frac{1}{2}\left\|\frac{\d\u(t)}{\d t}\right\|_{\H}^2+\||\u(t)||\nabla\u(t)|\|_{\H}^2+\|\f(t)\|_{\H}^2\nonumber\\&\leq \frac{1}{2}\left\|\frac{\d\u(t)}{\d t}\right\|_{\H}^2+\frac{\beta\mu}{2}\||\u(t)|^{\frac{r-1}{2}}|\nabla\u(t)|\|_{\H}^2+\vartheta\|\nabla\u(t)\|_{\H}^2.
\end{align*}
Integrating the above inequality from $\epsilon$ to $t$, we deduce  by using Ladyzhenskaya's inequality as 
\begin{align}\label{5p48}
	&	\int_{\epsilon}^t\left\|\frac{\d\u(s)}{\d t}\right\|_{\H}^2\d s+\mu\|\nabla\u(t)\|_{\H}^2+\alpha\|\u(t)\|_{\H}^2+\frac{2\beta}{r+1}\|\u(t)\|_{\wi\L^{r+1}}^{r+1}
	\nonumber\\&\leq \mu\|\nabla\u(\epsilon)\|_{\H}^2+\alpha\|\u(\epsilon)\|_{\H}^2+\frac{2\beta}{r+1}\|\u(\epsilon)\|_{\wi\L^{r+1}}^{r+1}+\beta\mu\int_{\epsilon}^t\||\u(s)|^{\frac{r-1}{2}}|\nabla\u(s)|\|_{\H}^2\d s
	\nonumber\\ & \quad +2\vartheta\int_0^t\|\nabla\u(s)\|_{\H}^2\d s
	\nonumber\\&\leq \frac{2\beta}{r+1}\|\u(\epsilon)\|_{\wi\L^{r+1}}^{r+1}+\left\{\begin{array}{cc}\left(\alpha+\frac{\vartheta}{\mu}+\frac{2}{\epsilon}\right)K+3\int_0^T\|\f(t)\|_{\H}^2\d t,&\text{ for }d=2,\\
		\left( \alpha+\frac{7\vartheta}{\mu}+\frac{2}{\epsilon}\right)K+6\int_0^T\|\f(t)\|_{\H}^2\d t,&\text{ for }d=3, \end{array}\right.
\end{align}
where we have used \eqref{02} and \eqref{05}. 
Note that by change of variable technique, we have
\begin{align}
	\int_0^t	\int_{\tau}^t\left\|\frac{\d\u(s)}{\d t}\right\|_{\H}^2\d s \d \tau =  \int_{0}^t s \left\|\frac{\d\u(s)}{\d t}\right\|_{\H}^2 \d s \geq \int_{\epsilon}^ts\left\|\frac{\d\u(s)}{\d t}\right\|_{\H}^2\d s.
\end{align} 
Therefore, integrating the inequality \eqref{5p48} over $\epsilon\in(0,t)$ results to 
\begin{align}\label{09}
		\int_{\epsilon}^ts\left\|\frac{\d\u(s)}{\d t}\right\|_{\H}^2\d s &\leq\frac{2\beta}{r+1}\int_0^t\|\u(\epsilon)\|_{\wi\L^{r+1}}^{r+1}\d\epsilon+\left\{\begin{array}{lc}\left(\alpha+\frac{\vartheta}{\mu}+\frac{2}{ \epsilon}\right)Kt+3t\int_0^T\|\f(t)\|_{\H}^2\d t,&\text{ for }d=2,\\
		\left( \alpha+\frac{7\vartheta}{\mu}+\frac{2}{\epsilon}\right)Kt+6t\int_0^T\|\f(t)\|_{\H}^2\d t,&\text{ for }d=3, \end{array}\right.\nonumber\\&\leq \left\{\begin{array}{lc}\left(\alpha+\frac{\vartheta}{\mu}+\frac{2}{ \epsilon}+\frac{1}{(r+1)t}\right)Kt+3t\int_0^T\|\f(t)\|_{\H}^2\d t,&\text{ for }d=2,\\
		\left( \alpha+\frac{7\vartheta}{\mu}+\frac{2}{ \epsilon}+\frac{1}{(r+1)t}\right)Kt+6t\int_0^T\|\f(t)\|_{\H}^2\d t,&\text{ for }d=3. \end{array}\right.
\end{align} 
Since 
$	\int_{\epsilon}^t\epsilon\left\|\frac{\d\u(s)}{\d t}\right\|_{\H}^2\d s\leq 	\int_{\epsilon}^ts\left\|\frac{\d\u(s)}{\d t}\right\|_{\H}^2\d s,$ from \eqref{09}, we also have 
\begin{align}\label{010}
	\int_{\epsilon}^t\left\|\frac{\d\u(s)}{\d t}\right\|_{\H}^2\d s\leq\frac{1}{\epsilon}\times \left\{\begin{array}{lc}\left(\alpha+\frac{\vartheta}{\mu}+\frac{2}{ \epsilon}+\frac{1}{(r+1)t}\right)Kt+3t\int_0^T\|\f(t)\|_{\H}^2\d t,&\text{ for }d=2,\\
		\left( \alpha+\frac{7\vartheta}{\mu}+\frac{2}{ \epsilon}+\frac{1}{(r+1)t}\right)Kt+6t\int_0^T\|\f(t)\|_{\H}^2\d t,&\text{ for }d=3. \end{array}\right.
\end{align}
so that $\frac{\d\u}{\d t}\in\mathrm{L}^2(\epsilon,T;\H)$ for any $\epsilon>0$. 
		Therefore, one can conclude that $\u\in\C((0,T];\V)\cap\mathrm{L}^2(\epsilon,T;\D(\A))\cap\mathrm{L}^{r+1}(\epsilon,T;\wi\L^{3(r+1)})$ for any $\epsilon>0$.

	Let us define $\v=\frac{\d \u}{\d t}$ and differentiating \eqref{555} with respect to $t$, we find that $\v(\cdot)$ satisfies 
		\begin{align}\label{324}
			\frac{\d  \v(t)}{\d  t}+\mu \A\v(t)+\B(\u(t),\v(t))+\B(\v(t),\u(t))+\alpha\v(t)+\beta\mathcal{C}'(\u(t))\v(t)&=\boldsymbol{f}_t,
		\end{align}
		for a.e. $t\in[\epsilon, T]$. Taking the inner product with $\v(\cdot)$ in the above equation, we obtain for a.e. $t\in[\epsilon, T]$
		\begin{align}\label{325}
			&	\frac{1}{2}\frac{\d}{\d t}\|\v(t)\|_{\H}^2+\mu\|\nabla\v(t)\|_{\H}^2+\alpha\|\v(t)\|_{\H}^2+\beta\||\u(t)|^{\frac{r-1}{2}}|\v(t)|\|_{\H}^2+(r-1)\beta\||\u(t)|^{\frac{r-3}{2}}(\u(t)\cdot\v(t))\|_{\H}^2\nonumber\\&=-(\B(\v(t),\u(t)),\v(t))+(\f_t(t),\v(t)) \nonumber\\&\leq\|\nabla\v(t)\|_{\H}\||\u(t)||\v(t)|\|_{\H}+\|\f_t(t)\|_{\H}\|\v(t)\|_{\H}\nonumber\\&\leq\frac{\mu}{2}\|\nabla\v(t)\|_{\H}^2+\frac{\beta}{2}\||\u(t)|^{\frac{r-1}{2}}|\v(t)|\|_{\H}^2+\frac{2\vartheta}{\mu}\|\v(t)\|_{\H}^2+\frac{\mu}{4\vartheta}\|\f_t(t)\|_{\H}^2,
		\end{align}
	where we have performed a calculation similar to \eqref{347}. 	Integrating the above inequality from $\epsilon_1$ to $t$, we deduce 
		\begin{align}\label{5p51}
			&\|\v(t)\|_{\H}^2+\mu\int_{\epsilon_1}^t\|\nabla\v(s)\|_{\H}^2\d s+2\alpha\int_{\epsilon_1}^t\|\v(s)\|_{\H}^2\d s+\beta\int_{\epsilon_1}^t\||\u(s)|^{\frac{r-1}{2}}|\v(s)|\|_{\H}^2\d s\nonumber\\&\quad+2(r-1)\beta\int_{\epsilon_1}^t\||\u(s)|^{\frac{r-3}{2}}(\u(s)\cdot\v(s))\|_{\H}^2\d s\nonumber\\&\leq\|\v(\epsilon_1)\|_{\H}^2+\frac{4\vartheta}{\mu}\int_{\epsilon_1}^t\|\v(s)\|_{\H}^2\d s+\frac{\mu}{2\vartheta}\int_{\e_1}^t\|\f_t(s)\|_{\H}^2\d s.
		\end{align}
		Once again integrating with respect to $\epsilon_1$ from $\epsilon$ to $t$ in \eqref{5p51}, we arrive at 
		\begin{align}\label{327}
			\|\v(t)\|_{\H}^2&\leq \frac{1}{(t-\epsilon)}\int_{\epsilon}^t\|\v(\epsilon_1)\|_{\H}^2\d\epsilon_1 +\frac{4\vartheta}{\mu} \int_{\epsilon}^t\|\v(s)\|_{\H}^2\d s+ \frac{\mu}{2\vartheta}\int_{\epsilon}^t\|\f_t(s)\|_{\H}^2\d s\nonumber\\&\leq \frac{\mu}{2\vartheta}\int_0^T\|\f_t(t)\|_{\H}^2\d t
			\nonumber\\&\quad+ \frac{1}{\epsilon}\left(\frac{1}{(t-\epsilon)}+\frac{4\vartheta}{\mu} \right) \times \left\{\begin{array}{lc}\left(\alpha+\frac{\vartheta}{\mu}+\frac{2}{ \epsilon}+\frac{1}{(r+1)t}\right)Kt+3t\int_0^T\|\f(t)\|_{\H}^2\d t,&\text{ for }d=2,\\
				\left( \alpha+\frac{7\vartheta}{\mu}+\frac{2}{ \epsilon}+\frac{1}{(r+1)t}\right)Kt+6t\int_0^T\|\f(t)\|_{\H}^2\d t,&\text{ for }d=3. \end{array}\right.
		\end{align}
		for all $0<\epsilon<\epsilon_1<t<T$, where we have used \eqref{010} also. Therefore,  we deduce from \eqref{5p51} that 
		\begin{align}\label{011}
		&\mu	\int_{\epsilon_1}^t\|\nabla\v(s)\|_{\H}^2\d s+2\alpha	\int_{\epsilon_1}^t\|\v(s)\|_{\H}^2\d s+\beta\int_{\epsilon_1}^t\||\u(s)|^{\frac{r-1}{2}}|\v(s)|\|_{\H}^2\d s\nonumber\\&\quad+2(r-1)\beta\int_{\epsilon_1}^t\||\u(s)|^{\frac{r-3}{2}}(\u(s)\cdot\v(s))\|_{\H}^2\d s\nonumber\\&\leq \frac{\mu}{2\vartheta} \int_0^T\|\f_t(t)\|_{\H}^2\d t\nonumber\\&\quad+\frac{1}{\epsilon}\left(\frac{1}{(t-\epsilon)}+\frac{4\vartheta}{\mu} \right) \times \left\{\begin{array}{lc}\left(\alpha+\frac{\vartheta}{\mu}+\frac{2}{ \epsilon}+\frac{1}{(r+1)t}\right)Kt+3t\int_0^T\|\f(t)\|_{\H}^2\d t,&\text{ for }d=2,\\
			\left( \alpha+\frac{7\vartheta}{\mu}+\frac{2}{ \epsilon}+\frac{1}{(r+1)t}\right)Kt+6t\int_0^T\|\f(t)\|_{\H}^2\d t,&\text{ for }d=3. \end{array}\right.
		\end{align}
	for all $0<\epsilon<\epsilon_1<t<T$. 
	
	For $d=3$ and $r\geq 5$, taking the inner product with $\frac{\d\v}{\d t}$ to the first equation in \eqref{324}, we find 
	\begin{align}\label{329}
	&	\left\|\frac{\d\v(t)}{\d t}\right\|_{\H}^2+\frac{\mu}{2}\frac{\d}{\d t}\|\nabla\v(t)\|_{\H}^2+\frac{\alpha}{2}\frac{\d}{\d t}\|\v(t)\|_{\H}^2+\frac{\beta}{2}\frac{\d}{\d t}\||\u(t)|^{\frac{r-1}{2}}\v(t)\|_{\H}^2\nonumber\\&\quad+\frac{\beta(r-1)}{2}\frac{\d}{\d t}\||\u(t)|^{\frac{r-3}{2}}|\u(t)\cdot\v(t)|\|_{\H}^2\nonumber\\&=-\left(\B(\u(t),\v(t)),\frac{\d\v(t)}{\d t}\right)-\left(\B(\v(t),\u(t)),\frac{\d\v(t)}{\d t}\right)+\frac{\beta}{2}\left(\frac{\d}{\d t}(|\u(t)|^{r-1}),|\v(t)|^2\right)\nonumber\\&\quad+\frac{\beta(r-1)}{2}\left(\frac{\d}{\d t}(|\u(t)|^{r-3}),|\u(t)\cdot\v(t)|^2\right)+\left(\f_t(t),\frac{\d\v(t)}{\d t}\right),
	\end{align}
for a.e. $t\in[0,T]$. We estimate $\left(\B(\u,\v),\frac{\d\v}{\d t}\right)+\left(\B(\v,\u),\frac{\d\v}{\d t}\right)$ using \eqref{b4}, H\"older's, Agmon's and Young's inequalities as 
\begin{align}\label{330}
	\left|\left(\B(\u,\v),\frac{\d\v}{\d t}\right)+\left(\B(\v,\u),\frac{\d\v}{\d t}\right)\right|
	&\leq \|\u\|_{\wi\L^{\infty}}\|\nabla\v\|_{\H}\left\|\frac{\d\v}{\d t}\right\|_{\H}+C\|\v\|_{\V}\|\u\|_{\V}^{1/2}\|\u\|_{\H^2}^{1/2}\left\|\frac{\d\v}{\d t}\right\|_{\H}
	\nonumber\\&\leq C\|\u\|_{\V}^{1/2}\|\u\|_{\H^2}^{1/2}\|\v\|_{\V}\left\|\frac{\d\v}{\d t}\right\|_{\H}
	\nonumber\\ & \leq\frac{1}{4}\left\|\frac{\d\v}{\d t}\right\|_{\H}^2+C\|\u\|_{\V}\|\u\|_{\H^2}\|\v\|_{\V}^2
	\nonumber\\ & \leq\frac{1}{4}\left\|\frac{\d\v}{\d t}\right\|_{\H}^2+C\|\u\|_{\V}\|\u\|_{\H}\|\v\|_{\V}^2 +C\|\u\|_{\V}\|\A\u\|_{\H}\|\v\|_{\V}^2
	\nonumber\\ & \leq\frac{1}{4}\left\|\frac{\d\v}{\d t}\right\|_{\H}^2+C\bigg[\|\u\|^2_{\H}+\|\nabla\u\|^2_{\H} +\|\A\u\|^2_{\H}\bigg]\|\v\|_{\V}^2. 
\end{align}
It can be easily seen that 
\begin{align}
	\left(\f_t,\frac{\d\v}{\d t}\right)\leq\|\f_t\|_{\H}\left\|\frac{\d\v}{\d t}\right\|_{\H}\leq\frac{1}{4}\left\|\frac{\d\v}{\d t}\right\|_{\H}^2+\|\f_t\|_{\H}^2.
\end{align}
We estimate the term $\frac{\beta}{2}\left(\frac{\d}{\d t}(|\u|^{r-1}),|\v|^2\right)+\frac{\beta(r-1)}{2}\left(\frac{\d}{\d t}(|\u|^{r-3}),|\u\cdot\v|^2\right),$ using H\"older's, interpolation, Gagliardo-Nirenberg's and Young's inequalities as 
\begin{align}\label{331}
&	\frac{\beta}{2}\left(\frac{\d}{\d t}(|\u|^{r-1}),|\v|^2\right)+\frac{\beta(r-1)}{2}\left(\frac{\d}{\d t}(|\u|^{r-3}),|\u\cdot\v|^2\right)\nonumber\\&=\frac{\beta(r-1)}{2}\left(|\u|^{r-3}(\u\cdot\v),|\v|^2\right)+\frac{\beta(r-1)(r-3)}{2}(|\u|^{r-5}(\u\cdot\v),|\u\cdot\v|^2)\nonumber\\&\leq \frac{\beta(r-1)}{2}\||\u|^{\frac{r-3}{2}}(\u\cdot\v)\|_{\H}\||\u|^{\frac{r-3}{2}}|\v|^2\|_{\H}\nonumber\\&\quad+\frac{\beta(r-1)(r-3)}{2}\||\u|^{\frac{r-3}{2}}(\u\cdot\v)\|_{\H}\||\u|^{\frac{r-5}{2}}|\v||\u\cdot\v|\|_{\H}\nonumber\\&\leq \frac{\beta(r-1)(r+1)}{2}\||\u|^{\frac{r-3}{2}}(\u\cdot\v)\|_{\H}\|\u\|_{\wi\L^{3(r+1)}}^{\frac{r-3}{2}}\|\v\|_{\wi\L^{\frac{6(r+1)}{r+3}}}^2\nonumber\\&\leq C\beta\||\u|^{\frac{r-3}{2}}(\u\cdot\v)\|_{\H}\|\u\|_{\wi\L^{3(r+1)}}^{\frac{r-3}{2}}\|\v\|_{\V}^{\frac{2r}{r+1}}\|\v\|_{\H}^{\frac{2}{r+1}}\nonumber\\&\leq \frac{\beta}{2}\||\u|^{\frac{r-3}{2}}(\u\cdot\v)\|_{\H}^2\|\v\|_{\V}^2+C\beta\|\u\|_{\wi\L^{3(r+1)}}^{r-3}\|\v\|_{\V}^{\frac{2(r-1)}{r+1}}\|\v\|_{\H}^{\frac{4}{r+1}}\nonumber\\&\leq \frac{\beta}{2}\||\u|^{\frac{r-3}{2}}(\u\cdot\v)\|_{\H}^2\|\v\|_{\V}^2+C\beta\|\u\|_{\wi\L^{3(r+1)}}^{\frac{(r-3)(r+1)}{r-1}}\|\v\|_{\V}^{2}+C\beta\|\v\|_{\H}^{2}.
\end{align}
Substituting \eqref{330} and \eqref{331} in \eqref{329} and then integrating from $\epsilon_2$ to $t$, we obtain 
\begin{align}\label{332}
	&	\int_{\epsilon_2}^t\left\|\frac{\d\v(s)}{\d t}\right\|_{\H}^2\d s+\mu\|\nabla\v(t)\|_{\H}^2+\alpha\|\v(t)\|_{\H}^2
	\nonumber\\ & \qquad +\beta\||\u(t)|^{\frac{r-1}{2}}|\v(t)|\|_{\H}^2+\beta(r-1)\||\u(t)|^{\frac{r-3}{2}}|\u(t)\cdot\v(t)|\|_{\H}^2\nonumber\\&\leq \mu\|\v(\epsilon_2)\|_{\V}^2+\alpha\|\v(\epsilon_2)\|_{\H}^2+\beta\||\u(\epsilon_2)|^{\frac{r-1}{2}}|\v(\epsilon_2)|\|_{\V}^2+\beta(r-1)\||\u(\epsilon_2)|^{\frac{r-3}{2}}|\u(\epsilon_2)\cdot\v(\epsilon_2)|\|_{\H}^2\nonumber\\&\quad+2\int_{\epsilon_2}^t\|\f_t(s)\|_{\H}^2\d s+C\int_{\epsilon_2}^t\bigg[\|\u(s)\|^2_{\V}+\|\A\u(s)\|^2_{\H}\bigg]\|\v(s)\|_{\V}^2\d s
	\nonumber\\ & \quad +\beta\int_{\epsilon_2}^t\||\u(s)|^{\frac{r-3}{2}}(\u(s)\cdot\v(s))\|_{\H}^2\|\v(s)\|_{\V}^2\d s\nonumber\\&\quad +C\beta\int_{\epsilon_2}^t\|\u(s)\|_{\wi\L^{3(r+1)}}^{\frac{(r-3)(r+1)}{r-1}}\|\v(s)\|_{\V}^{2}\d s+C\beta\int_{\epsilon_2}^t\|\v(s)\|_{\H}^{2}\d s,
\end{align}
for all $t\in[\epsilon_1,T]$. Integrating the above inequality over $\epsilon_2\in(\epsilon_1,t)$, we have 
\begin{align}\label{333}
	&\mu\|\nabla\v(t)\|_{\H}^2+\alpha\|\v(t)\|_{\H}^2+\beta\||\u(t)|^{\frac{r-1}{2}}|\v(t)|\|_{\H}^2+\beta(r-1)\||\u(t)|^{\frac{r-3}{2}}|\u(t)\cdot\v(t)|\|_{\H}^2\nonumber\\&\leq\frac{1}{\epsilon_1}\bigg[\mu\int_{\epsilon_1}^t\|\v(\epsilon_2)\|_{\V}^2\d\epsilon_2+\alpha\int_{\epsilon_1}^t\|\v(\epsilon_2)\|_{\H}^2\d\epsilon_2+\beta\int_{\epsilon_1}^t\||\u(\epsilon_2)|^{\frac{r-1}{2}}|\v(\epsilon_2)|\|_{\V}^2\d\epsilon_2\nonumber\\&\quad+\beta(r-1)\int_{\epsilon_1}^t\||\u(\epsilon_2)|^{\frac{r-3}{2}}|\u(\epsilon_2)\cdot\v(\epsilon_2)|\|_{\H}^2\d\epsilon_2\bigg]+2\int_{\epsilon_1}^t(s-\epsilon_1)\|\f_t(s)\|_{\H}^2\d s \nonumber\\&\quad+C\int_{\epsilon_1}^t(s-\epsilon_1)\bigg[\|\u(s)\|_{\V}^2+\|\A\u(s)\|^2_{\H}\bigg]\|\v(s)\|_{\V}^2\d s \nonumber\\ & \quad +\beta\int_{\epsilon_1}^t(s-\epsilon_1)\||\u(s)|^{\frac{r-3}{2}}(\u(s)\cdot\v(s))\|_{\H}^2\|\v(s)\|_{\V}^2\d s \nonumber\\&\quad+C\beta\int_{\epsilon_1}^t(s-\epsilon_1)\|\v(s)\|_{\H}^{2}\d s +C\beta\int_{\epsilon_1}^t(s-\epsilon_1)\|\u(s)\|_{\wi\L^{3(r+1)}}^{\frac{(r-3)(r+1)}{r-1}}\|\v(s)\|_{\V}^{2}\d s\nonumber\\&\leq\frac{1}{\epsilon}\left\{\frac{4}{\mu}\int_0^T\|\f_t(t)\|_{\V'}^2\d t +\frac{1}{\epsilon}\left(\frac{1}{(t-\epsilon)}+\frac{4\vartheta}{\mu} \right)\left( \alpha+\frac{7\vartheta}{\mu}+\frac{2}{\epsilon}+\frac{1}{(r+1)t}\right)Kt+6\int_0^T\|\f(t)\|_{\H}^2\d t\right\} \nonumber\\&\quad+2T\int_0^T\|\f_t(t)\|_{\H}^2\d t+CT\int_{\epsilon_1}^t\bigg[\|\u(s)\|_{\V}^2+\|\A\u(s)\|^2_{\H}\bigg]\|\v(s)\|_{\V}^2\d s \nonumber\\&\quad+\beta T\int_{\epsilon_1}^t\||\u(s)|^{\frac{r-3}{2}}(\u(s)\cdot\v(s))\|_{\H}^2\|\v(s)\|_{\V}^2\d s +C\beta T\int_{\epsilon_1}^t\|\u(s)\|_{\wi\L^{3(r+1)}}^{\frac{(r-3)(r+1)}{r-1}}\|\v(s)\|_{\V}^{2}\d s,
\end{align}
	for all $0<\epsilon<\epsilon_1<t<T$, where we have used \eqref{011} also. An application of Gronwall's inequality in \eqref{333} yields 
	\begin{align}\label{334}
	& \min\{\mu,\alpha\} \|\v(t)\|^2_{\V}  \nonumber\\&\leq\bigg[\frac{1}{\epsilon}\left\{\frac{4}{\mu}\int_0^T\|\f_t(t)\|_{\V'}^2\d t +\frac{1}{\epsilon}\left(\frac{1}{(t-\epsilon)}+\frac{4\vartheta}{\mu} \right)\left( \alpha+\frac{8\vartheta}{\mu}+\frac{2}{ t}+\frac{1}{(r+1)t}\right)Kt+6\int_0^T\|\f(t)\|_{\H}^2\d t\right\} 
		\nonumber\\&\quad+2T\int_0^T\|\f_t(t)\|_{\H}^2\d t\bigg]\exp\bigg\{CT^2\sup_{t\in[\epsilon_1,t]}\|\u(t)\|^2_{\V}+ CT\int_{\epsilon_1}^t\|\A\u(s)\|_{\H}^2\d s \nonumber\\&\qquad+\beta T\int_{\epsilon_1}^t\||\u(s)|^{\frac{r-3}{2}}|\u(s)\cdot\v(s)|\|_{\H}^2\d s+C\beta T^{\frac{r+1}{r-1}}\left(\int_{\epsilon_1}^t\|\u(s)\|_{\wi\L^{3(r+1)}}^{r+1}\d s\right)^{\frac{r-3}{r-1}}\bigg\}\leq C,
	\end{align}
	for all $0<\epsilon<\epsilon_1<t<T,$ and the right hand side is bounded by using \eqref{16}, \eqref{05} and \eqref{011}.

			\vskip 0.2cm
		\noindent 	\textbf{Step 3:} \textsl{Backward uniqueness:} 
		Let us now prove the backward uniqueness property. Let us first consider the case $d=2,3$ and $r\in(3,\infty)$.	Let $\u_1(\cdot)$ and $\u_2(\cdot)$ be two solutions of the system \eqref{555} with the same final data, say $\boldsymbol{\xi}$ and forcing $\f\in\W^{1,2}(0,T;\H)$.  Then $\u=\u_1-\u_2$ satisfies the following system in $\V'+\wi\L^{\frac{r+1}{r}}$ for a.e. $t\in[0,T]$ and in $\H$ for a.e. $t\in[\epsilon_1,T]$, for any $0<\epsilon<\epsilon_1<T$:
		\begin{align}\label{336}
			\left\{
			\begin{aligned}
				\frac{\d  \u(t)}{\d  t}+\mathcal{A}(\u(t))&=-[\B(\u_1(t),\u(t))+\B(\u(t),\u_2(t))]=:h(\u(t)), \\
				\u(T)&=\boldsymbol{0},
			\end{aligned}\right. 
		\end{align}
		where 
		\begin{align}
			\mathcal{A}(\u):=\mu\A\u+\alpha\u+\beta[\mathcal{C}(\u_1)-\mathcal{C}(\u_2)]=\mu\A\u+\alpha\u+\beta\int_0^1\mathcal{C}'(\theta\u+\u_2)\u\d\theta.
		\end{align}
	By the monotonicity of $\mathcal{C}(\cdot)$ (see \eqref{2.23}), we know that 
		\begin{align}\label{014}
			\langle\mathcal{A}(\u),\u\rangle &=\mu\|\nabla\u\|_{\H}^2+\alpha\|\u\|_{\H}^2+\langle\mathcal{C}(\u_1)-\mathcal{C}(\u_2),\u\rangle \geq \min\{\mu,\alpha\}\|\u\|_{\V}^2. 
		\end{align}
	If we denote $\partial_t\u=\v$, then 
	\begin{align*}
		\partial_t[\mathcal{A}(\u)]=\mu\A\v+\alpha\v+\beta\int_0^1\mathcal{C}'(\theta\u+\u_2)\v\d\theta+\beta\int_0^1\mathcal{C}''(\theta\u+\u_2)(\u\otimes(\theta\v+\partial_t\u_2))\d\theta. 
	\end{align*}
		It should be noted that 
		\begin{align}\label{19}
			\langle\partial_t[\mathcal{A}(\u)],\u\rangle=\langle\mathcal{A}(\u),\v\rangle+\beta\left\langle\int_0^1\mathcal{C}''(\theta\u+\u_2)(\u\otimes(\theta\v+\partial_t\u_2))\d\theta,\u\right\rangle,
		\end{align}
	since \begin{align*}
		\langle\mathcal{A}(\u),\v\rangle&=\langle\mu\A\u,\v\rangle+\alpha(\u,\v)+\beta\int_{\mathbb{T}^d}\int_0^1|\theta\u+\u_2|^{r-1}\d\theta\u\cdot\v\d x\nonumber\\&\quad+\beta(r-1)\int_{\mathbb{T}^d}\int_0^1|\theta\u+\u_2|^{r-3}((\theta\u+\u_2)\cdot\v)((\theta\u+\u_2)\cdot\u)\d\theta\d x\nonumber\\&=\left\langle\u,\mu\A\v+\alpha\v+\beta\int_{\mathbb{T}^d}\int_0^1|\theta\u+\u_2|^{r-1}\d\theta\v\d x\right.\nonumber\\&\qquad\left.+\beta(r-1)\int_{\mathbb{T}^d}\int_0^1(\theta\u+\u_2)|\theta\u+\u_2|^{r-3}((\theta\u+\u_2)\cdot\v)\d\theta\d x\right\rangle\nonumber\\&=\left\langle\u,\mu\A\v+\alpha\v+\beta\int_0^1\mathcal{C}'(\theta\u+\u_2)\v\d\theta\right\rangle.
	\end{align*}
		Our aim is to show that if $\u(T)=\boldsymbol{0}$, then $\u(t)=\boldsymbol{0}$ for all $t\in[0,T]$. We prove this result by a contradiction first in the interval $[\epsilon_1,T,]$ for any $\epsilon_1>0$ and then by using  the continuity of $\|\u(t)\|_{\H}$ in $[0,T]$ and  the arbitrariness of $\epsilon_1$, one can obtain the required result in $[0,T]$. Assume that there exists some $t_0\in[\epsilon_1,T)$ such that $\u(t_0)\neq \boldsymbol{0}$. Since the mapping $t\mapsto\|\u(t)\|_{\H}$ is continuous, the following alternative holds: 
		\begin{enumerate}
			\item [(i)] for all $t\in[t_0,T]$, $\|\u(t)\|_{\H}>0$ or
			\item [(ii)] there exists a $t_1\in(t_0,T)$ such that for all $t\in(t_0,t_1)$, $\|\u(t)\|_{\H}>0$ and $\u(t_1)=\boldsymbol{0}$. 
		\end{enumerate}
		In the second case, denote by $\Lambda(t)$, the ratio 
		\begin{align}\label{339}
			\Lambda(t)=\frac{\langle\mathcal{A}(\u(t)),\u(t)\rangle}{\|\u(t)\|_{\H}^2}\geq \min\{\mu,\alpha\} \frac{\|\u(t)\|_{\V}^2}{\|\u(t)\|_{\H}^2},
		\end{align}
where we have used \eqref{014}. The ratio $\frac{\|\u\|_{\V}^2}{\|\u\|_{\H}^2}$ is known as ``Dirichlet's quotient'' (\cite{JCR5}).	Therefore, we have the following equation for $\Lambda(t)$:
		\begin{align}\label{33}
			\frac{\d\Lambda}{\d t}&=\frac{\langle\partial_t[\mathcal{A}(\u)],\u\rangle}{\|\u\|_{\H}^2}+\frac{\langle\mathcal{A}(\u)-2\Lambda\u,\partial_t\u\rangle}{\|\u\|_{\H}^2}\nonumber\\&=\frac{\beta}{\|\u\|_{\H}^2}\left\langle\int_0^1\mathcal{C}''(\theta\u+\u_2)(\u\otimes(\theta\partial_t\u+\partial_t\u_2))\d\theta,\u\right\rangle+2\frac{\langle\mathcal{A}(\u)-\Lambda\u,\partial_t\u\rangle}{\|\u\|_{\H}^2},
		\end{align}
		where we have used \eqref{19} also. 	Since $\langle\mathcal{A}(\u)-\Lambda\u,\u\rangle=0$ and $\partial_t\u=-\mathcal{A}(\u)+h(\u)$, it follows that 
		\begin{align}\label{016}
			&	\frac{1}{2}\frac{\d\Lambda}{\d t}+\frac{\|\mathcal{A}(\u)-\Lambda\u\|_{\H}^2}{\|\u\|_{\H}^2}\nonumber\\&=\frac{\beta}{2\|\u\|_{\H}^2}\left\langle\int_0^1\mathcal{C}''(\theta\u+\u_2)(\u\otimes(\theta\partial_t\u_1+(1-\theta)\partial_t\u_2))\d\theta,\u\right\rangle+\frac{\langle \mathcal{A}(\u)-\Lambda\u,h(\u)\rangle}{\|\u\|_{\H}^2}\nonumber\\&\leq\frac{\beta}{2\|\u\|_{\H}^2}\int_0^1\|\mathcal{C}''(\theta\u_1+(1-\theta)\u_2)(\u\otimes(\theta\partial_t\u_1+(1-\theta)\partial_t\u_2))\|_{\V'}\d\theta\|\u\|_{\V}\nonumber\\&\quad+\frac{1}{2}\frac{\|\mathcal{A}(\cdot)\u-\Lambda\u\|_{\H}^2}{\|\u\|_{\H}^2}+\frac{1}{2}\frac{\|h(\u)\|_{\H}^2}{\|\u\|_{\H}^2}. 
		\end{align}
		For $d=2$, $r\in(3,\infty)$ and $d=3$, $r\in(3,5]$,	we estimate  the term $$\|\mathcal{C}''(\theta\u_1+(1-\theta)\u_2)(\u\otimes(\theta\partial_t\u+\partial_t\u_2))\|_{\V'}$$ from \eqref{016} using H\"older's and Sobolev's inequalities  as 
		\begin{align}\label{017}
		&\|\mathcal{C}''(\theta\u_1+(1-\theta)\u_2)(\u\otimes(\theta\partial_t\u_1+(1-\theta)\partial_t\u_2))\|_{\V'}\nonumber\\&\leq\|\mathcal{C}''(\theta\u_1+(1-\theta)\u_2)(\u\otimes\partial_t\u_1)\|_{\V'}+\|\mathcal{C}''(\theta\u_1+(1-\theta)\u_2)(\u\otimes\partial_t\u_2)\|_{\V'}\nonumber\\&\leq C\|\mathcal{C}''(\theta\u_1+(1-\theta)\u_2)(\u\otimes\partial_t\u_1)\|_{\wi\L^{\frac{r+1}{r}}}+C\|\mathcal{C}''(\theta\u_1+(1-\theta)\u_2)(\u\otimes\partial_t\u_2)\|_{\wi\L^{\frac{r+1}{r}}}
			\nonumber\\&\leq C\left(\|\u_1\|_{\wi\L^{r+1}}+\|\u_2\|_{\wi\L^{r+1}}\right)^{r-2}\|\u\|_{\wi\L^{r+1}}\|\partial_t\u_1\|_{\wi\L^{r+1}}+C\left(\|\u_1\|_{\wi\L^{r+1}}+\|\u_2\|_{\wi\L^{r+1}}\right)^{r-2}\|\u\|_{\wi\L^{r+1}}\|\partial_t\u_2\|_{\wi\L^{r+1}}\nonumber\\&\leq C\left(\|\u_1\|_{\wi\L^{r+1}}^{r-2}+\|\u_2\|_{\wi\L^{r+1}}^{r-2}\right)\left(\|\partial_t\u_1\|_{\V}+\|\partial_t\u_2\|_{\V}\right)\|\u\|_{\V},
		\end{align}
		for all $0<\theta<1$. For $d=3$ and $r\in(5,\infty)$, we estimate $\|\mathcal{C}''(\theta\u_1+(1-\theta)\u_2)(\u\otimes(\theta\partial_t\u+\partial_t\u_2))\|_{\V'}$ using Gagliardo-Nirenberg's H\"older's and interpolation inequalities as 
		\begin{align}\label{017a}
		&\|\mathcal{C}''(\theta\u_1+(1-\theta)\u_2)(\u\otimes(\theta\partial_t\u_1+(1-\theta)\partial_t\u_2))\|_{\V'}\nonumber\\	&\leq C\|\mathcal{C}''(\theta\u_1+(1-\theta)\u_2)(\u\otimes\partial_t\u_1)\|_{\wi\L^{\frac{6}{5}}}+C\|\mathcal{C}''(\theta\u_1+(1-\theta)\u_2)(\u\otimes\partial_t\u_2)\|_{\wi\L^{\frac{6}{5}}}
			\nonumber\\&\leq C\left(\|\u_1\|_{\wi\L^{2(r-2)}}+\|\u_2\|_{\wi\L^{2(r-2)}}\right)^{r-2}\|\u\|_{\wi\L^6}\|\partial_t\u_1\|_{\wi\L^6}+C\left(\|\u_1\|_{\wi\L^{2(r-2)}}+\|\u_2\|_{\wi\L^{2(r-2)}}\right)^{r-2}\|\u\|_{\wi\L^6}\|\partial_t\u_2\|_{\wi\L^6}\nonumber\\&\leq C\left(\|\u_1\|_{\wi\L^{r+1}}^{\frac{r+7}{4}}\|\u_1\|_{\wi\L^{3(r+1)}}^{\frac{3(r-5)}{4}}+\|\u_2\|_{\wi\L^{r+1}}^{\frac{r+7}{4}}\|\u_2\|_{\wi\L^{3(r+1)}}^{\frac{3(r-5)}{4}}\right)\left(\|\partial_t\u_1\|_{\V}+\|\partial_t\u_2\|_{\V}\right)\|\u\|_{\V}.
		\end{align}
		For $d=2$, using H\"older's,  Agmon's and Young's inequalities, and \eqref{b2}, we have  \begin{align}\label{018}
			\|h(\u)\|_{\H}^2&\leq\|\B(\u_1,\u)+\B(\u,\u_2)\|_{\H}^2\nonumber\\&\leq C\|\u_1\|_{\wi\L^{\infty}}^2\|\u\|_{\V}^2+C\|\u\|_{\H}\|\u\|_{\V}\|\u_2\|_{\H}\|\u_2\|_{\H^2}\nonumber\\&\leq C\|\u_1\|_{\H}\|\u_1\|_{\H^2}\|\u\|_{\V}^2+C\|\u_2\|_{\H}\|\u_2\|_{\H^2}\|\u\|_{\V}^2.
		\end{align}
		For $d=3$, a similar calculation yields (see \eqref{b4})
		\begin{align}\label{019}
			\|h(\u)\|_{\H}^2&\leq C\left(\|\u_1\|_{\V}\|\u_1\|_{\H^2}+C\|\u_2\|_{\V}\|\u_2\|_{\H^2}\right)\|\u\|_{\V}^2.
		\end{align}
		Therefore, for $d=2$, $r\in(3,\infty)$ and $d=3$, $r\in(3,5]$,		using \eqref{017}-\eqref{019} in \eqref{016}, we deduce 
		\begin{align}\label{4p7}
				\frac{\d\Lambda}{\d t}+\frac{\|\mathcal{A}(\u)-\Lambda\u\|_{\H}^2}{\|\u\|_{\H}^2} &\leq C\beta\left(\|\u_1\|_{\wi\L^{r+1}}^{r-2}+\|\u_2\|_{\wi\L^{r+1}}^{r-2}\right)\left(\|\partial_t\u_1\|_{\V}+\|\partial_t\u_2\|_{\V}\right)\frac{\|\u\|_{\V}^2}{\|\u\|_{\H}^2}\nonumber\\&\quad+C\frac{\|\u\|_{\V}^2}{\|\u\|_{\H}^2}\left\{\begin{array}{ll}\left(\|\u_1\|_{\H}\|\u_1\|_{\H^2}+C\|\u_2\|_{\H}\|\u_2\|_{\H^2}\right)&\text{ for }d=2,\\ \left(\|\u_1\|_{\V}\|\u_1\|_{\H^2}+C\|\u_2\|_{\V}\|\u_2\|_{\H^2}\right)&\text{ for }d=3,
			\end{array}\right.\nonumber\\&\leq \frac{C\beta}{\min\{\mu,\alpha\}}\Lambda\left(\|\u_1\|_{\wi\L^{r+1}}^{r-2}+\|\u_2\|_{\wi\L^{r+1}}^{r-2}\right)\left(\|\partial_t\u_1\|_{\V}+\|\partial_t\u_2\|_{\V}\right)\nonumber\\&\quad+\frac{C}{\min\{\mu,\alpha\}}\Lambda\left\{\begin{array}{ll}\left(\|\u_1\|_{\H}\|\u_1\|_{\H^2}+C\|\u_2\|_{\H}\|\u_2\|_{\H^2}\right)&\text{ for }d=2,\\ \left(\|\u_1\|_{\V}\|\u_1\|_{\H^2}+C\|\u_2\|_{\V}\|\u_2\|_{\H^2}\right)&\text{ for }d=3,\end{array}\right.
		\end{align}
		where we have used \eqref{014} also. The variation of constants formula results to
		\begin{align}\label{020}
			\Lambda(t)\leq\Lambda(t_0)\left\{\begin{array}{ll}\exp\left[\frac{C\beta T^{\frac{1}{2}}}{\min\{\mu,\alpha\}}\sum\limits_{i=1}^2\sup\limits_{t\in[t_0,T]}\|\u_i(t)\|_{\V}^{r-2}\left(\int_{t_0}^T\|\partial_t\u_i(t)\|_{\V}^2\d t\right)^{1/2}\right]&\\\times\exp\left[\frac{CT^{\frac{1}{2}}}{\min\{\mu,\alpha\}}\sum\limits_{i=1}^2\sup\limits_{t\in[0,T]}\|\u_i(t)\|_{\H}\left(\int_{t_0}^T\|\u_i(t)\|_{\H^2}^2\d t\right)^{1/2}\right],&\text{ for }d=2, \\
				\exp\left[\frac{C\beta T^{\frac{1}{2}}}{\min\{\mu,\alpha\}}\sum\limits_{i=1}^2\sup\limits_{t\in[t_0,T]}\|\u_i(t)\|_{\V}^{r-2}\left(\int_{t_0}^T\|\partial_t\u_i(t)\|_{\V}^2\d t\right)^{1/2}\right]&\\\times\exp\left[\frac{CT^{\frac{1}{2}}}{\min\{\mu,\alpha\}}\sum\limits_{i=1}^2\sup\limits_{t\in[t_0,T]}\|\u_i(t)\|_{\V}\left(\int_{t_0}^T\|\u_i(t)\|_{\H^2}^2\d t\right)^{1/2}\right],&\text{ for }d=3. \end{array}\right. 
		\end{align}
	Using the estimates \eqref{02}, \eqref{16}, \eqref{05} and \eqref{011}, one can easily see that the right hand side of \eqref{020} is finite. 
	
	For $d=3$ and $r\in(5,\infty)$, by using the estimates \eqref{017a} and \eqref{019} in \eqref{016}, we  have 
			\begin{align*}
				\frac{\d\Lambda}{\d t}&\leq\frac{C\Lambda}{\min\{\mu,\alpha\}}\left(\|\u_1\|_{\V}\|\u_1\|_{\H^2}+C\|\u_2\|_{\V}\|\u_2\|_{\H^2}\right)\nonumber\\&\quad+\frac{C\beta\Lambda}{\min\{\mu,\alpha\}}\left(\|\u_1\|_{\wi\L^{r+1}}^{\frac{r+7}{4}}\|\u_1\|_{\wi\L^{3(r+1)}}^{\frac{3(r-5)}{4}}+\|\u_2\|_{\wi\L^{r+1}}^{\frac{r+7}{4}}\|\u_2\|_{\wi\L^{3(r+1)}}^{\frac{3(r-5)}{4}}\right)\|\partial_t\u\|_{\V}.
		\end{align*}
	Once again using the variation of constants formula, we arrive at 
	\begin{align}\label{3p47}
		\Lambda(t)&\leq\Lambda(t_0)\exp 	\left[\frac{C\beta T^{\frac{r+19}{4(r+1)}}}{\min\{\mu,\alpha\}}\sum\limits_{i=1}^2\sup\limits_{t\in[t_0,T]}\|\u_i(t)\|_{\wi\L^{r+1}}^{\frac{r+7}{4}}\left(\int_{t_0}^T\|\u_i(t)\|_{\wi\L^{3(r+1)}}^{r+1}\d t\right)^{\frac{3(r-5)}{4(r+1)}}\sup_{t\in[t_0,T]}\|\partial_t\u_i(t)\|_{\V}\right]\nonumber\\&\quad\times\exp\left[\frac{CT^{\frac{1}{2}}}{\min\{\mu,\alpha\}}\sum\limits_{i=1}^2\sup\limits_{t\in[t_0,T]}\|\u_i(t)\|_{\V}\left(\int_{t_0}^T\|\u_i(t)\|_{\H^2}^2\d t\right)^{1/2}\right].
	\end{align}
Using the estimates \eqref{03}, \eqref{05}, \eqref{011} and \eqref{334}, one can easily see that the right hand side of \eqref{3p47} is finite.

		On the other hand, we have 
		\begin{align}\label{023}
			-\frac{1}{2\|\u(t)\|_{\H}^2}\frac{\d}{\d t}\|\u(t)\|_{\H}^2=-\frac{\langle\u(t),\partial_t\u(t)\rangle}{\|\u(t)\|_{\H}^2}=\Lambda(t)-\frac{\langle h(\u(t)),\u(t)\rangle}{\|\u(t)\|_{\H}^2}. 
		\end{align}
		An application of H\"older's, Ladyzhenskaya's and Young's inequalities, and \eqref{339}  yield
		\begin{align}\label{43}
			\frac{|\langle h(\u),\u\rangle|}{\|\u\|_{\H}^2}&\leq \frac{\|h(\u)\|_{\V'}\|\u\|_{\V}}{\|\u\|_{\H}^2}\leq\frac{\left(\|\u_1\|_{\wi\L^4}+\|\u_2\|_{\wi\L^4}\right)\|\u\|_{\wi\L^4}\|\u\|_{\V}}{\|\u\|_{\H}^2}\nonumber\\&\leq\frac{2^{\frac{d-1}{4}}\left(\|\u_1\|_{\wi\L^4}+\|\u_2\|_{\wi\L^4}\right)\|\u\|_{\V}^{1+\frac{d}{4}}\|\u\|_{\H}^{1-\frac{d}{4}}}{\|\u\|_{\H}^{2}}\leq 2^{\frac{d-1}{4}}\left(\|\u_1\|_{\wi\L^4}+\|\u_2\|_{\wi\L^4}\right)\left(\frac{\Lambda}{\min\{\mu,\alpha\}}\right)^{\frac{4+d}{8}} \nonumber\\&\leq\Lambda+\frac{C}{\min\{\mu,\alpha\}^{\frac{4+d}{4-d}}}\left(\|\u_1\|_{\wi\L^4}^{\frac{8}{4-d}}+\|\u_2\|_{\wi\L^4}^{\frac{8}{4-d}}\right). 
		\end{align}
		Therefore, from \eqref{023}, we easily have 
		\begin{align}\label{025}
			-\frac{\d}{\d t}\log\|\u(t)\|_{\H}\leq 2\Lambda(t)+\frac{C}{\min\{\mu,\alpha\}^{\frac{4+d}{4-d}}}\left(\|\u_1(t)\|_{\wi\L^4}^{\frac{8}{4-d}}+\|\u_2(t)\|_{\wi\L^4}^{\frac{8}{4-d}}\right). 
		\end{align}
		According to \eqref{020} and \eqref{16}, the right hand side of \eqref{025} is integrable on $(t_0,t_1)$ and this contradicts the fact that $\u(t_1)=\boldsymbol{0}$. Thus we are in the case (i) of the alternative and backward uniqueness result follows. 
		
For the case, $d=3$ and $2\beta\mu\geq 1$, one can show the backward uniqueness in a similar way as in the previous case.

For the case, $d=2$ and $r\in[1,3]$, we slightly modify the system \eqref{336} and consider 
	\begin{equation}\label{353}
		\left\{
		\begin{aligned}
			\frac{\d  \u(t)}{\d  t}+\widetilde{\mathcal{A}}\u(t)&=-\bigg[\B(\u_1(t),\u(t))+\B(\u(t),\u_2(t))\\&\quad+\beta\int_0^1\mathcal{C}'(\theta\u_1(t)+(1-\theta)\u_2(t))\d\theta\u(t)\bigg]=:\widetilde{h}(\u(t)), \\
			\u(T)&=\boldsymbol{0},
		\end{aligned}\right. 
	\end{equation}
where $\widetilde{\mathcal{A}}\u=\mu\A\u+\alpha\u$. 	In the this case, we denote $\widetilde{\Lambda}(t)$, the ratio 
\begin{align*}
	\widetilde{\Lambda}(t)=\frac{\langle\widetilde{\mathcal{A}}\u(t),\u(t)\rangle}{\|\u(t)\|_{\H}^2}=\frac{\mu\|\nabla\u(t)\|_{\H}^2+\alpha\|\u(t)\|_{\H}^2}{\|\u(t)\|_{\H}^2}\geq \min\{\mu,\alpha\}\frac{\|\u(t)\|_{\V}^2}{\|\u(t)\|_{\H}^2}.
\end{align*}
Calculations similar to \eqref{33} and \eqref{016} provide
\begin{align}\label{46}
	\frac{1}{2}\frac{\d\widetilde{\Lambda}}{\d t}+\frac{\|\widetilde{\mathcal{A}}\u-\widetilde{\Lambda}\u\|_{\H}^2}{\|\u\|_{\H}^2}&=\frac{\langle\widetilde{\mathcal{A}}\u-\widetilde{\Lambda}\u,\widetilde{h}(\u)\rangle}{\|\u\|_{\H}^2}\leq\frac{1}{2}\frac{\|\widetilde{\mathcal{A}}\u-\widetilde{\Lambda}\u\|_{\H}^2}{\|\u\|_{\H}^2}+\frac{1}{2}\frac{\|\widetilde{h}(\u)\|_{\H}^2}{\|\u\|_{\H}^2}. 
\end{align} 
A bound similar to \eqref{020} can be obtained by demonstrating bounds akin to \eqref{018} and \eqref{43}. Using H\"older's, Gagliardo-Nirenberg's, Sobolev's and Young's inequalities,  we find 
\begin{align}\label{47}
\left\|\int_0^1\mathcal{C}'(\theta\u_1+(1-\theta)\u_2)\d\theta\u\right\|_{\H}^2&\leq C\left(\|\u_1\|_{\wi\L^{2(r+1)}}^{2(r-1)}+\|\u_2\|_{\wi\L^{2(r+1)}}^{2(r-1)}\right) \|\u\|_{\wi\L^{r+1}}^2\nonumber\\&\leq C\left(\|\u_1\|_{\V}^{2(r-1)}+\|\u_2\|_{\V}^{2(r-1)}\right) \|\u\|_{\V}^2,
\end{align}
and 
\begin{align}\label{48}
\frac{1}{\|\u\|_{\H}^2}\left|	\left\langle\int_0^1\mathcal{C}'(\theta\u_1+(1-\theta)\u_2)\d\theta\u,\u\right\rangle\right| &\leq\frac{1}{\|\u\|_{\H}^2}\left\|\int_0^1\mathcal{C}'(\theta\u_1+(1-\theta)\u_2)\d\theta\u\right\|_{\V'}\|\u\|_{\V}\nonumber\\&\leq \frac{C}{\|\u\|_{\H}^2}\left(\|\u_1\|_{\wi\L^{r+1}}^{r-1}+\|\u_2\|_{\wi\L^{r+1}}^{r-1}\right)\|\u\|_{\wi\L^{r+1}}\|\u\|_{\V}\nonumber\\&\leq\frac{C}{\|\u\|_{\H}^2}\left(\|\u_1\|_{\wi\L^{r+1}}^{r-1}+\|\u_2\|_{\wi\L^{r+1}}^{r-1}\right)\|\u\|_{\V}^{\frac{2r}{r+1}}\|\u\|_{\H}^{\frac{2}{r+1}}\nonumber\\&\leq\widetilde{\Lambda}+\frac{C}{\min\{\mu,\alpha\}^{r+1}}
\left(\|\u_1\|_{\wi\L^{r+1}}^{(r+1)(r-1)}+\|\u_2\|_{\wi\L^{r+1}}^{(r+1)(r-1)}\right),
\end{align}
and one can complete the proof as in the previous case. We point out that the same method can be used to prove for the case $d=2$, $r\in(3,\infty)$ and $d=3$, $r\in[3,5]$ ($2\beta\mu\geq 1$ for $d=r=3$) also (see \eqref{49} below), but not for $d=3,r\in(5,\infty)$. The primary benefit of this method is that we only need the regularity   $\u\in\C((0,T];\V)\cap\mathrm{L}^2(\epsilon,T;\D(\A))\cap\mathrm{L}^{r+1}(\epsilon,T;\wi\L^{3(r+1)})$.  
	\end{proof}
	
	\begin{remark}
		In one considers 2D NSE ($\alpha=\beta=0$ in \eqref{1}), then one can assume zero mean condition  and take advantage of the Poincar\'e inequality to establish  the backward uniqueness results  (cf. \cite{IK,JCR5,RT1}). 
		\end{remark}

\begin{corollary}\label{cor3.2}
	Under the assumption of Theorem \ref{thm3.1},  either  $\u$ vanishes identically or $\u$ 	never vanishes. 
\end{corollary}
\begin{proof}
	Combining the forward uniqueness proved in Step 1 and the backward uniqueness  proved in Step 3  of the proof of Theorem \ref{thm3.1}, one can conclude the proof. 
\end{proof}

	\section{Applications}\label{Sec4}\setcounter{equation}{0}
	As a consequence of the backward uniqueness result, we first establish the approximate controllability result with respect to the initial data. Then we use the backward uniqueness in attractor theory to show that solution map is injective. Furthermore, we  establish a crucial estimate on the global attractor  which helps to prove the zero Lipschitz deviation of  the global attractor. The uniqueness as well as continuity with respect to the Eulerian initial data  of Lagrangian trajectories in 2D and 3D CBF flows are also examined in this section. 
	\subsection{Approximate conrollability}
		As an immediate consequence of the backward uniqueness result, we obtain the \emph{approximate controllability with respect to the initial data,  which is viewed as a start controller} (cf. \cite{VBMR,JLL}, etc.) for 2D and 3D CBF equations.  We follow the works \cite{VBMR,CFJP}, etc. to obtain our required results.    
		\begin{definition}
			Let $T>0$. The system \eqref{555} is \emph{approximately controllable with respect to the initial data} in time $T$ if, for every $\u_1\in\H$, and for every $\epsilon>0$, there exists an $\x\in\H$ such that the solution $\u^{\x}$ of the problem \eqref{555} satisfies $\|\u^{\x}(T)-\u_1\|_{\H}\leq\epsilon.$
		\end{definition}
		Corollary \ref{cor3.2} implies the  lack of null-controllability of the model \eqref{555} with respect to the initial data, asserting that there are no non-trivial initial  data that can be steered to zero. 
		
		In order to establish the approximate controllability with respect to initial data, we need the following result from  \cite{CBLT}.
	\begin{lemma}[Proposition IV.1, \cite{CBLT}]\label{lem3.1}
		Let $\S(t )$ be a family of continuous nonlinear operators in $\H$. We assume that for a $\u_0\in\H$, the map $\x\mapsto\S(t)(\x)$  is Fr\'echet differentiable at the point $\u_0$ and  denote its Fr\'echet derivative by $D\S(t)(\u_0)$.  If the operator  $(D \S(t) (\u_0))^*$ is injective, then the subspace generated by $\S(t) (\H)$ is dense in $\H$.
	\end{lemma}
	\begin{theorem}[Approximate controllability]\label{thm3.3}
		The space $\{\u^{\boldsymbol{x}}(T):\boldsymbol{x}\in\H\}$ is dense in $\H$, where $\u^{\boldsymbol{x}}(\cdot)$ is the unique solution to the  system \eqref{555}. 
	\end{theorem}

\begin{proof}
	Let us define the semiflow (nonlinear semigroup) $\mathrm{S}(t):\H\to\H$ by $$\mathrm{S}(t)(\boldsymbol{x})=\u^{\boldsymbol{x}}(t), \ t\in[0,T],$$  where $\u^{\boldsymbol{x}}(\cdot)$ is the unique solution to the  system \eqref{555}. It can be easily seen that $\mathrm{S}(T)$ is Fr\'echet differentiable on $\H$ and its Fr\'echet  derivative $\Gamma:\H\to\H$  is given by $\Gamma\boldsymbol{y}=D\mathrm{S}(T)(\boldsymbol{x})\boldsymbol{y}=\boldsymbol{v}(T)$, where $\v\in\C([0,T];\H)$ is the unique solution of the equation
	\begin{equation}\label{556}
		\left\{
		\begin{aligned}
			\frac{\d  \v(t)}{\d  t}+\mu \A\v(t)+\B'(\u(t))\v(t)+\alpha\v(t)+\beta\mathcal{C}'(\u(t))\v(t)&=\boldsymbol{0}, \\
			\v(0)&=\boldsymbol{y},
		\end{aligned}\right. 
	\end{equation} 
	where $\langle\B'(\u)\v,\w\rangle=\langle\B(\u,\v)+\B(\v,\u),\w\rangle=b(\u,\v,\w)+b(\v,\u,\w),$ and 
$\mathcal{C}'(\cdot)$ is defined in \eqref{2p9}. As the system \eqref{556} is linear and the base state $\u(\cdot)$ has the regularity results discussed in the proof  of Theorem \ref{thm3.1}, the existence and uniqueness of weak solution of the system can be obtained by using a standard Faedo-Galerkin method  (see \cite{MT3}).  Then, the dual operator $\Gamma^*:\H\to\H$  is given by $\Gamma^*\p=\z(0)$, for all $\p\in\H$, where $\z(\cdot)$ is the solution to the following backward dual equation:
	\begin{equation}\label{557}
		\left\{
		\begin{aligned}
			-	\frac{\d  \z(t)}{\d  t}+\mu \A\z(t)+(\B'(\u(t)))^*\z(t)+\alpha\z(t)+\beta\mathcal{C}'(\u(t))\z(t)&=\boldsymbol{0}, \\
			\z(T)&=\boldsymbol{p},
		\end{aligned}\right. 
	\end{equation} 
	where $\langle(\B'(\u))^*\v,\w\rangle=\langle\v,\B'(\u)\w\rangle=b(\v,\u,\w)+b(\v,\w,\u),$ for $\u,\v,\w\in\V$.  The system \eqref{557} is well-posed for all $\p\in\H$ (see \cite{MT3}). The operator $\Gamma^*$ is injective on $\H$ is a  consequence of the following backward uniqueness result (Theorem \ref{thm3.1}): 
	\begin{equation}
		\left\{
		\begin{aligned}
			-	\frac{\d  \z(t)}{\d  t}+\mu \A\z(t)+(\B'(\u(t)))^*\z(t)+\alpha\z(t)+\beta\mathcal{C}'(\u(t))\z(t)&=\boldsymbol{0}, \\
			\z(0)&=\boldsymbol{0},
		\end{aligned}\right\}\Rightarrow\z\equiv \boldsymbol{0}.
	\end{equation} 
	If  $\mathcal{N}(\Gamma^*)$ and $\mathcal{R}(\Gamma)$ denote the null space and range space of the operators $\Gamma^*$ and $\Gamma$, respectively, then $\mathcal{N}(\Gamma^*)=\mathcal{R}(\Gamma)^{\perp}$. Suppose that $\u_1\in\H$  is orthogonal to the range of $\Gamma $ in $\H$. We solve the system \eqref{557} with $\z(T)=\u_1$. Let $\v(\cdot)$ be a solution of \eqref{556} with $\v(0)=\u_0$  for any $\u_0\in\H$. Multiplying by $\v(\cdot)$, the equation satisfied by $\z(\cdot)$, we get 
	\begin{align*}
		(\v(T),\u_1)=(\u_0,\z(0)), \ \text{ for all }\ \u_0\in\H. 
	\end{align*}
	Since $\u_1$ is orthogonal to $\v(T)$ in $\H$, we deduce that $$(\u_0,\z(0))=0,\ \text{ for all }\ \u_0\in\H.$$ But then $\z(0)=\boldsymbol{0}$ in $\H$  and by the backward uniqueness result $\z\equiv \boldsymbol{0}$ in $\H$ and $\u_1=\boldsymbol{0}$. The injectivity of the operator $\Gamma^*$ in $\H$ and Proposition IV.1, \cite{CBLT} (see Lemma \ref{lem3.1}) imply that the space  $\{\mathrm{S}(T)(\x):\boldsymbol{x}\in\H\}$ is dense in $\H$, and the approximate controllability result follows. 
\end{proof}

	\subsection{The Lipschitz deviation} In this section, we discuss an another consequence of the backward uniqueness result in connection with the global attractors for 2D and 3D CBF equations (cf.  \cite{IK,JCR5} for 2D NSE). We prove that the  Lipschitz deviation  for the global  attractor	of the 2D and 3D  CBF equations is zero. Similar results for 2D NSE have been obtained in \cite{EPJC,JCR5} and we follow these works to establish the required result.
	
	Inspired from \cite{BRVY}, the authors in \cite{EJJCR} defined a new	quantity called \emph{Lipschitz deviation},  which measures how well a compact set $\mathrm{X}$ in a Hilbert space $(\mathcal{H},\|\cdot\|_{\mathcal{H}})$ can be approximated by graphs of Lipschitz functions (with prescribed Lipschitz constant) defined over a finite dimensional subspace of $\mathcal{H}$.  The \emph{Hausdorff semidistance} between two non-empty subsets $\mathrm{X}, \mathrm{Y} \subseteq\mathcal{H}$ is defined by $$\mathrm{dist}(\mathrm{X},\mathrm{Y}):=\sup\limits_{x\in\mathrm{X}}\inf\limits_{y\in\mathrm{Y}}\|x-y\|_{\mathcal{H}}.$$ 
	\begin{definition}[Definition 2.1, \cite{EPJC}]
		Let $\mathrm{X}$ be a compact subset of a real Hilbert space $\mathcal{H}$. Let $\delta_m(\mathrm{X},\e)$ be the smallest dimension of a linear subspace $\mathrm{U}\subset\mathcal{H}$  such that $$\mathrm{dist}(\mathrm{X},\mathrm{G}_{\mathrm{U}}[\varphi])<\e,$$  for some $m$-Lipschitz function $\varphi:\mathrm{U}\to\mathrm{U}^{\perp}$, that is, $$\|\varphi(u)-\varphi(v)\|_{\mathcal{H}}\leq m\|u-v\|_{\mathcal{H}}, \ \text{ for all }\ u,v\in\mathrm{U},$$ where $\mathrm{U}^{\perp}$ is the orthogonal complement of $\mathrm{U}$ in $\mathcal{H}$ and $\mathrm{G}_{\mathrm{U}}[\varphi]$ is the graph of $\varphi$ over $\mathrm{U}$: $$\mathrm{G}_{\mathrm{U}}[\varphi]=\{u+\varphi(u):u\in\mathrm{U}\}.$$ The $m$-Lipschitz deviation of $\mathrm{X}$, $\mathrm{dev}_m(\mathrm{X})$, is given by $$\mathrm{dev}_m(\mathrm{X})=\limsup\limits_{\e\to 0}\frac{\log\delta_m(\mathrm{X},\e)}{-\log\e}.$$ The  Lipschitz deviation of $\mathrm{X}$, $\mathrm{dev}(\mathrm{X})$, is given by $$\mathrm{dev}(\mathrm{X})=\lim\limits_{m\to\infty}\mathrm{dev}_m(\mathrm{X}).$$ 
		\end{definition}
	
	From \cite[Theorem 3.5]{KKMTM} (see \cite{KKMTM_RMP} for 2D CBF equations), we know that under the condition $\f\in\H$,  independent of $t$, the system \eqref{555}  possesses a global attractor
	\begin{align*}
		\mathscr{A}=\left\{\x\in\H:\S(t)(\x)\text{ exists  for all }t\in\mathbb{R}, \text{ and }\sup\limits_{t\in\R}\|\S(t)(\x)\|_{\H}<\infty \right\},
	\end{align*}
	where $\S(t)(\x)$ denotes the solution to the system \eqref{555} starting at $\x$. 
	\begin{lemma}\label{lem3.7}
		The semigroup $\S(t):\H\to\H$ is injective, for every $t>0$.
	\end{lemma}
\begin{proof}
	Suppose that $\x,\y\in\H$ and $\u(T)=\v(T),$ that is, $\S(T)(\x)=\S(T)(\y)$, for some $T>0$. Then by Theorem \ref{thm3.1}, we know that $\S(t)(\x)=\S(t)(\y)$ for all $t\in[0,T]$ and in particular $\x=\y$ and hence the injectivity of $\S(t):\H\to\H$ follows. 
\end{proof}
If the semigroup $\S(t)$ is injective on $\mathscr{A}$, then the dynamics, restricted to $\mathscr{A}$, actually define a dynamical system, that is, $\S(t)\big|_{\mathscr{A}}$ makes sense for all $t\in\R$, not just for $t\geq 0$ and $\S(t)\mathscr{A}=\mathscr{A}$ for all $t\in\mathbb{R}$ (Theorem 13, \cite{JCR5}). 

The following results for 2D NSE can be found in  \cite[Theorem 3.1]{IK}, \cite[Theorem 5.1]{EPJC1}, \cite[Proposition 43]{JCR5} and   \cite[Theorem 13.3]{JCR1}. 
	\begin{theorem}\label{thm3.5}
	For $d=2$, $r\in[1,\infty)$ and $d=3$, $r\in[3,\infty)$ ($2\beta\mu\geq 1$ for $d=r=3$),	under the above assumptions, we have 
		\begin{align}\label{360}
			\sup\limits_{\u_1,\u_2\in\mathscr{A},\u_1\neq\u_2}\frac{\|\u_1-\u_2\|_{\V}^2}{\|\u_1-\u_2\|_{\H}^2\log\left(\frac{M_0^2}{\|\u_1-\u_2\|_{\H}^2}\right)}<\infty,
		\end{align}where $M_0\geq 4\sup\limits_{\x\in\H}\|\x\|_{\H}$. 
	\end{theorem}
	\begin{proof}
	  For any $\x\in\H$, from \cite[Theorem 3.4]{KKMTM}, we infer that $\|\S(t)(\x)\|_{\V}\leq M_1$ for sufficiently large $t$. By using the estimates similar to  \eqref{327} and \eqref{334} on time derivative, one can show further that 
$\|\A\S(t)(\x)\|_{\H}\leq M_2$ for sufficiently large $t$ (see \eqref{382} below and  \cite[Propsoition 12.4]{JCR_2001} for 2D NSE).  In particular, $\mathscr{A}$ is bounded in $\D(\A)$. From the definition of $M_0$, it is clear that $\log\left(\frac{M_0^2}{\|\u_1-\u_2\|_{\H}^2}\right)\geq \log 4\geq 1$. 

Let us now consider a ``log-Dirichlet's quotient" (cf. \cite{IK})
\begin{align*}
	\widetilde{\Q}(t)=\frac{\widetilde{\Lambda}(t)}{\log\left(\frac{M_0^2}{\|\u(t)\|_{\H}^2}\right)}. 
\end{align*}
Then, differentiating $\widetilde{\Q}(t)$ with respect to $t$ and then using \eqref{4p7}, \eqref{43} and \eqref{48}, we find 
\begin{align}\label{361}
	\frac{\d\widetilde{\Q}}{\d t}&=\frac{\log\left(\frac{M_0^2}{\|\u\|_{\H}^2}\right)\frac{\d\widetilde{\Lambda}}{\d t}+\widetilde{\Lambda}\frac{\d}{\d t}\log(\|\u\|_{\H}^2)}{\left[\log\left(\frac{M_0^2}{\|\u(t)\|_{\H}^2}\right)\right]^2}=\frac{\log\left(\frac{M_0^2}{\|\u\|_{\H}^2}\right)\frac{\d\widetilde{\Lambda}}{\d t}+\widetilde{\Lambda}\frac{2}{\|\u\|_{\H}^2}\langle\partial_t\u,\u\rangle}{\left[\log\left(\frac{M_0^2}{\|\u(t)\|_{\H}^2}\right)\right]^2}\nonumber\\&\leq k_1(t)\widetilde{\Q}+\frac{-2\widetilde{\Lambda}^2+2\widetilde{\Lambda} \frac{\left|\langle h(\u),\u\rangle\right|}{\|\u\|_{\H}^2}}{\left[\log\left(\frac{M_0^2}{\|\u(t)\|_{\H}^2}\right)\right]^2}\leq k_1(t)\widetilde{\Q}-\widetilde{\Q}^2+k_2(t)\frac{\widetilde{\Q}}{\log\left(\frac{M_0^2}{\|\u(t)\|_{\H}^2}\right)}\nonumber\\&\leq -\widetilde{\Q}^2+(k_1(t)+k_2(t))\widetilde{\Q},
\end{align}
where 
\begin{align*}
	k_1&= \left\{\begin{array}{ll}\frac{C}{\min\{\mu,\alpha\}}\left(\|\u_1\|_{\H}\|\u_1\|_{\H^2}+C\|\u_2\|_{\H}\|\u_2\|_{\H^2}\right)+\frac{C\beta}{\min\{\mu,\alpha\}}\left(\|\u_1\|_{\V}^{2(r-1)}+\|\u_2\|_{\V}^{2(r-1)}\right),&\text{for }d=2,\\ \frac{C}{\min\{\mu,\alpha\}}\left(\|\u_1\|_{\V}\|\u_1\|_{\H^2}+C\|\u_2\|_{\V}\|\u_2\|_{\H^2}\right)&\\\quad+\frac{C\beta}{\min\{\mu,\alpha\}}\left(\|\u_1\|_{\wi\L^{r+1}}^{\frac{r-1}{2}}\|\u_1\|_{\wi\L^{3(r+1)}}^{\frac{3(r-1)}{2}}+\|\u_2\|_{\wi\L^{r+1}}^{\frac{r-1}{2}}\|\u_2\|_{\wi\L^{3(r+1)}}^{\frac{3(r-1)}{2}}\right),&\text{for }d=3,\end{array}\right.\\
	k_2&=\left\{\begin{array}{ll}\frac{C}{\min\{\mu,\alpha\}^{\frac{4+d}{4-d}}}\left(\|\u_1\|_{\wi\L^4}^{\frac{8}{4-d}}+\|\u_2\|_{\wi\L^4}^{\frac{8}{4-d}}\right)+\frac{C}{\min\{\mu,\alpha\}^{r+1}}\left(\|\u_1\|_{\wi\L^{r+1}}^{(r+1)(r-1)}+\|\u_2\|_{\wi\L^{r+1}}^{(r+1)(r-1)}\right),& \\ \hspace{6cm} \text{ for }d=2,r\in[1,\infty) \ \text{ and }\ d=3, r\in[3,5], &\\
	\frac{C}{\min\{\mu,\alpha\}^{7}}\left(\|\u_1\|_{\wi\L^4}^{8}+\|\u_2\|_{\wi\L^4}^{8}\right) \hspace{1.25cm} \text{ for }d=3,r\in[5,\infty).\end{array}\right.
	\end{align*}
By defining $\upsilon=\widetilde{\Q}^{-1}$, from \eqref{361}, we deduce 
\begin{align*}
	\frac{\d\upsilon}{\d t}+(k_1(t)+k_2(t))\upsilon\geq 1. 
	\end{align*}
An application of the variation of constants formula yields
\begin{align}\label{362}
	\widetilde{\Q}(t)\leq \frac{\widetilde{\Q}(0)\exp\left(\int_{0}^t(k_1(s)+k_2(s))\d s\right)}{1+\widetilde{\Q}(0)\int_{0}^t\exp\left(\int_{0}^s(k_1(r)+k_2(r))\d r\right)\d s}. 
\end{align}
As $t\to\infty$, the right hand side of $\widetilde{\Q}(t)$ in \eqref{362} tends to $k_1(t)+k_2(t)\leq C(M_1,M_2)$ (since $\D(\A)\subset\L^p$ for any $p\in[2,\infty)$). Therefore, we obtain that there exists $T$ such that
\begin{align*}
	\widetilde{\Q}(t)\leq C(M_1,M_2)\ \text{ for all }\ t\geq T,
\end{align*}
where $C(M_1,M_2)$ and $T$ are independent of $\widetilde{\Q}(0)$. 

Let us now consider $\x,\y\in\mathscr{A}$. Since solutions in the attractor exist for all time, we know there exists $t\geq T$ such that $\x=\S(t)(\u_1(-t))$  and $\y=\S(t)(\u_2(-t))$ with $\x\neq \y$. Since $\S(\cdot)$ injective (Lemma \ref{lem3.7}), $\u_1(-t)\neq\u_2(-t)$. Moreover, $\widetilde{\Q}(-t)<\infty$ implies that $\widetilde{\Q}(0)\leq C(M_1,M_2)$. Finally, we have 
\begin{align*}
\sup_{\x,\y\in\mathscr{A},\ \x\neq\y}	\widetilde{\Q}(t)\leq C(M_1,M_2),
\end{align*}
so that \eqref{360} follows. 
\end{proof}

It should be noted  that whether \eqref{360} holds without the factor $\log\left(\frac{M_0^2}{\|\u_1-\u_2\|_{\H}^2}\right)$ is an open problem. The above result can be used to obtain the $1$-log-Lipschitz continuity of $\A : \mathscr{A}\to\H$ (cf.  \cite[Corollary 5.2]{EPJC1} for 2D NSE). 
\begin{corollary}\label{cor3.9}
	There exists a constant $\widehat{K}>0$ such that 
	\begin{align}\label{3p63}
		\|\A(\u_1-\u_2)\|_{\H}\leq \widehat{K}\|\u_1-\u_2\|_{\H}\log\left(\frac{\widehat{M}_0^2}{\|\u_1-\u_2\|_{\H}^2}\right),\ \text{ for all }\ \u_1,\u_2\in\mathscr{A},\ \u_1\neq\u_2,
	\end{align}
for some $\widehat{M}_0\geq 4\sup\limits_{\x\in\H}\|\A^{\frac{1}{2}}\x\|_{\H}$. 
\end{corollary}
\begin{proof}
	We use the fact that $\mathscr{A}$ is bounded in $\D(\A)$. Let us now consider $\widetilde{h}(\u)$ from \eqref{353} and estimate $\|\A^{\frac{1}{2}}\widetilde{h}(\u)\|_{\H}$ using fractional Leibniz rule (\cite[Theorem 1]{LGSO}), and Sobolev's inequality as
	\begin{align}
	&	\|\A^{\frac{1}{2}}\widetilde{h}(\u)\|_{\H}\nonumber\\&\leq\|\A^{\frac{1}{2}}\B(\u_1,\u)\|_{\H}+\|\A^{\frac{1}{2}}\B(\u,\u_2)\|_{\H}+\beta\left\|\A^{\frac{1}{2}}\left(\int_0^1\mathcal{C}'(\theta\u_1+(1-\theta)\u_2)\d\theta\u\right)\right\|_{\H}\nonumber\\&\leq C\|\A(\u_1\otimes\u)\|_{\H}+\|\A(\u\otimes\u_2)\|_{\H}+C\beta\|\A^{\frac{1}{2}}(|\u_1|^{r-1}\u)\|_{\H}+C\beta\|\A^{\frac{1}{2}}(|\u_2|^{r-1}\u)\|_{\H}\nonumber\\&\leq C\left(\|\A\u_1\|_{\H}+\|\A\u_2\|_{\H}\right)\|\u\|_{\wi\L^{\infty}}+C\left(\|\u_1\|_{\wi\L^{\infty}}+\|\u_2\|_{\wi\L^{\infty}}\right)\|\A\u\|_{\H}\nonumber\\&\quad+C\beta\left(\|\u_1\|_{\wi\L^{\infty}}^{r-1}+\|\u_2\|_{\wi\L^{\infty}}^{r-1}\right)\|\A^{\frac{1}{2}}\u\|_{\H}+C\beta\left(\|\u_1\|_{\wi\L^{\infty}}^{r-2}\|\A^{\frac{1}{2}}\u_1\|_{\H}+\|\u_2\|_{\wi\L^{\infty}}^{r-2}\|\A^{\frac{1}{2}}\u_2\|_{\H}\right)\|\u\|_{\wi\L^{\infty}}\nonumber\\&\leq C(M_1,M_2)\|\A\u\|_{\H},
	\end{align}
for sufficiently large $t$. Therefore,  calculations similar to the proof of Theorem \ref{thm3.5} yield  for all $\u_1,\u_2\in\mathscr{A}$ with $ \u_1\neq\u_2,$
\begin{align}\label{365}
	\|\A(\u_1-\u_2)\|_{\H}^2\leq C_1\|\A^{\frac{1}{2}}(\u_1-\u_2)\|_{\H}^2\log\left(\frac{\widetilde{M}_0^2}{\|\A^{\frac{1}{2}}(\u_1-\u_2)\|_{\H}^2}\right),
\end{align}
where $\widetilde{M}_0\geq 4\sup\limits_{\x\in\H}\|\A^{\frac{1}{2}}\x\|_{\H}$ and $C_1$ is a constant. Therefore, combining \eqref{1.5} and \eqref{365}, we obtain  
\begin{align*}
	\|\A(\u_1-\u_2)\|_{\H}^2\leq C_0C_1\|\u_1-\u_2\|_{\H}^2\log\left(\frac{{M}_0^2}{\|\u_1-\u_2\|_{\H}^2}\right)\log\left(\frac{\widehat{M}_0^2}{\|\A^{\frac{1}{2}}(\u_1-\u_2)\|_{\H}^2}\right).
\end{align*}
Since $\|\u_1-\u_2\|_{\H}\leq\frac{1}{\sqrt{\lambda_1}}\|\A^{\frac{1}{2}}(\u_1-\u_2)\|_{\H}$, we further have 
\begin{align}
	\|\A(\u_1-\u_2)\|_{\H}^2\leq C_0C_1\|\u_1-\u_2\|_{\H}^2\log\left(\frac{{M}_0^2}{\|\u_1-\u_2\|_{\H}^2}\right)\log\left(\frac{\widetilde{M}_0^2}{\lambda_1\|\u_1-\u_2\|_{\H}^2}\right).
\end{align}
One can choose $M_0,\widetilde{M}_0$ and $\widehat{M}_0$ such that $M_0\leq \frac{\widetilde{M}_0}{\sqrt{\lambda_1}}\leq\widehat{M}_0$. Therefore, we finally derive \eqref{3p63} for  $\widehat{K}=\sqrt{C_0C_1}$. 
\end{proof}

Using either Theorem \ref{thm3.5} or Corollary \ref{cor3.9}, it follows from \cite[Proposition 4.2]{EPJC1} (or Corollary 4.4) that there exists a family of approximating Lipschitz manifolds $\mathcal{M}_N$ ($N=2M_0e^{\frac{-\lambda_{n+1}}{2C_0}}$, see \eqref{364} below), such that the global attractor $\mathscr{A}$ associated with  the 2D and 3D  CBF equations lies within an exponentially small neighbourhood of $\mathcal{M}_N$ and hence has zero Lipschitz deviation.  
\begin{theorem}\label{thm3.6}
For $d=2$, $r\in[1,\infty)$ and $d=3$, $r\in[3,\infty)$ ($2\beta\mu\geq 1$ for $d=r=3$),	if $\f\in\H$, then $\mathrm{dev} (\mathscr{A})=0$, where $\mathscr{A}$ is the global	attractor for the 2D and 3D  CBF equations. 
\end{theorem}
\begin{proof}
	Let $\mathrm{P}_n$ be the orthogonal projection onto the first $n$ eigenfunctions of the Stokes operator $\A$ and $\mathrm{Q}_n=\mathrm{I}-\mathrm{P}_n$ so that $$\mathrm{P}_n\x=\sum\limits_{k=1}^n(\x,\boldsymbol{e}_k)\boldsymbol{e}_k\ \text{ and }\ \mathrm{Q}_n\x=\sum\limits_{k=n+1}^{\infty}(\x,\boldsymbol{e}_k)\boldsymbol{e}_k.$$ Consider a subset $\mathrm{X}$ of $\mathscr{A}$ that is maximal	for the relation 
	\begin{align}\label{363}
		\|\mathrm{Q}_n(\x-\y)\|_{\H}\leq 	\|\mathrm{P}_n(\x-\y)\|_{\H}, \ \text{ for all }\ \x,\y\in\mathrm{X}. 
	\end{align}
For every $\boldsymbol{p}\in\mathrm{P}_n\mathrm{X}$ with $\boldsymbol{p}=\mathrm{P}_n\x,\ \x\in\mathrm{X},$ define $\phi_n(\boldsymbol{p})=\mathrm{Q}_n\x.$ From \eqref{363}, this is well-defined and 
\begin{align*}
	\|\phi_n(\boldsymbol{p})-\phi_n(\boldsymbol{p}')\|_{\H}\leq \|\boldsymbol{p}-\boldsymbol{p}'\|_{\H}, \ \text{ for all }\ \boldsymbol{p},\boldsymbol{p}'\in\mathrm{P}_n\mathrm{X}. 
\end{align*}
Since $\mathrm{X}$ is a closed subset of the compact set $\mathscr{A}$, $\mathrm{P}_n\mathrm{X}$ is closed,  using \cite[Theorem 12.3]{JHLR}, we can extend $\phi_n$  from the closed set $\mathrm{P}_n\mathrm{X}\subset\mathrm{P}_n\H$ to a function $\Phi:\mathrm{P}_n\H\to\mathrm{Q}_n\H$, preserving the Lipschitz constant. Our aim is to show that 
\begin{align}\label{364}
	\mathrm{dist}(\mathscr{A},\mathrm{G}_{\mathrm{P}_n\H}[\Phi])\leq\e_n=2M_0e^{\frac{-\lambda_{n+1}}{2C_0}},
\end{align}
where $\{\lambda_n\}_{n=1}^{\infty}$ is the set of eigenvalues of $\A$. Indeed, if $\x\in\mathscr{A}$ but $\x\not\in\mathrm{X}$, then there is a $\y\in\mathrm{X}$ such that 
\begin{align}
	\|\mathrm{Q}_n(\x-\y)\|_{\H}\geq 	\|\mathrm{P}_n(\x-\y)\|_{\H}.
\end{align}
Setting $\w=\x-\y$, we have  $\|\mathrm{Q}_n\w\|_{\H}\leq \|\w\|_{\H}\leq\|\mathrm{P}_n\w\|_{\H}+\|\mathrm{Q}_n\w\|_{\H}\leq 2\|\mathrm{Q}_n\w\|_{\H}$ and 
\begin{align*}
	\|\A^{\frac{1}{2}}\w\|_{\H}^2=\sum_{k=1}^{\infty}\lambda_k|(\w,\boldsymbol{e}_k)|^2\geq \sum_{k=n+1}^{\infty}\lambda_k|(\w,\boldsymbol{e}_k)|^2\geq\lambda_{n+1}\|\mathrm{Q}_n\w\|_{\H}^2. 
\end{align*}
Therefore, using \eqref{1.5}, we deduce 
\begin{align}
\lambda_{n+1}\|\mathrm{Q}_n\w\|_{\H}^2\leq\|\A^{\frac{1}{2}}\w\|_{\H}^2\leq 4C\|\mathrm{Q}_n\w\|_{\H}^2\log\left(\frac{M_0^2}{\|\mathrm{Q}_n\w\|_{\H}^2}\right),
\end{align}
so that $\|\Q_n\w\|_{\H}\leq \frac{\e_n}{2}=M_0e^{\frac{-\lambda_{n+1}}{2C_0}},$ and \eqref{364} follows. The estimate \eqref{364} implies that $\delta_1(\mathscr{A},\e_n)=n$ and hence 
\begin{align}
	\limsup\limits_{n\to\infty}\frac{\log\delta_1(\mathscr{A},\e_n)}{-\log\e_n}=	\limsup\limits_{n\to\infty}\frac{\log n}{\frac{\lambda_{n+1}}{2C_0}-\log(2M_0)}.
\end{align}
Since the  eigenvalues $\lambda_n\sim \lambda_1n^{\frac{d}{2}}$ as $n\to\infty$ (\cite[Page 54]{FMRT}),  we further have 
\begin{align*}
	\limsup\limits_{n\to\infty}\frac{\log n}{\lambda_{n}}=0. 
	\end{align*}
Therefore, $\mathrm{dev}(\mathscr{A})\leq \mathrm{dev}_1(\mathscr{A})=0.$ 
\end{proof}

\subsection{The uniqueness of Lagrangian trajectories in 2D and 3D CBF flows} Let us now prove  the uniqueness of Lagrangian trajectories in 2D and 3D CBF flows by using the log-Lipschitz regularity. Given an initial data  in $\H^{d-2}_p(\T^d)$ ($\H^{2}_p(\T^3)$ for $d=3$ and $r\in(5,\infty)$), the 2D and 3D CBF  equations \eqref{1} have a unique solution. Given such a solution, we show  that the Lagrangian particle trajectories are also unique. More precisely, the question is whether the solutions of the ordinary differential equation 
\begin{equation}\label{372}
	{\boldsymbol{X}}'=\u(\boldsymbol{X},t), \
	\boldsymbol{X}(0)=\boldsymbol{X}_0,
\end{equation}
are unique, when $\u(t)$ is a solution of the CBF  with initial data $\u_0\in\H^{d-2}_p(\T^d)$. For $\boldsymbol{W}=\boldsymbol{X}-\boldsymbol{Y}$, a calculation similar to \eqref{023} yields 
\begin{align}\label{373}
-\frac{\d}{\d t}\log|\boldsymbol{W}(t)|&=	-\frac{1}{2|\boldsymbol{W}(t)|^2}\frac{\d}{\d t}|\boldsymbol{W}(t)|^2=-\frac{(\boldsymbol{W}(t),{\boldsymbol{W}}'(t))}{|\boldsymbol{W}(t)|^2}\leq\frac{|\u(\boldsymbol{X},t)-\u(\boldsymbol{Y},t)|}{|\boldsymbol{W}(t)|}\nonumber\\&=\frac{1}{|\boldsymbol{W}(t)|}\left|\int_0^1\nabla\u(\theta\boldsymbol{X}+(1-\theta\boldsymbol{Y}),t)\d\theta\cdot\boldsymbol{W}(t)\right|\nonumber\\&\leq C\|\nabla\u(t)\|_{\L^{\infty}}\leq C\|\u(t)\|_{\H^s}, \ \text{ for }\ s>\frac{d}{2}+1,
\end{align}
where we have used the Sobolev's inequality. Integrating the above inequality from $t_0$ to $t>t_0$, one obtains 
\begin{align}\label{374}
	-\log|\boldsymbol{W}(t)|\leq -\log|\boldsymbol{W}(t_0)|+\int_{t_0}^t\|\u(r)\|_{\H^s}\d r, \ \text{ for }\ s>\frac{d}{2}+1. 
\end{align}
If  $\u\in\mathrm{L}^1(0,T;\H^s_p(\T^d))$,  for $s>\frac{d}{2}+1$, then one could put $t_0 = 0$ in \eqref{374} and immediately obtain uniqueness. For the critical case $s=\frac{d}{2}+1$, by using  the log-Lipschitz regularity (see \cite[Theorem 2]{EZ}),  we know that $$|\u(\boldsymbol{X},t)-\u(\boldsymbol{Y},t)|\leq C\|\u(t)\|_{\H^{\frac{d}{2}+1}}|\boldsymbol{X}-\boldsymbol{Y}|\left(-\log|\boldsymbol{X}-\boldsymbol{Y}|\right)^{\frac{1}{2}}.$$ Therefore, from \eqref{373}, we immediately have 
\begin{align}
	-\frac{\d}{\d t}[\left(-\log|\boldsymbol{W}(t)|\right)^{\frac{1}{2}}]^2\leq C\|\u(t)\|_{\H^{\frac{d}{2}+1}}\left(-\log|\boldsymbol{W}(t)\right)^{\frac{1}{2}}.
\end{align}
Integrating the above inequality from $t_0$ to $t>t_0$, we find  
\begin{align}\label{375}
	-\left(-\log|\boldsymbol{W}(t)|\right)^{\frac{1}{2}}\leq 	-\left(-\log|\boldsymbol{W}(t_0)|\right)^{\frac{1}{2}}+C\int_{t_0}^t\|\u(r)\|_{\H^{\frac{d}{2}+1}}\d r.
\end{align}
Thus, if  $\u\in\mathrm{L}^1(0,T;\H^{\frac{d}{2}+1}_p(\T^d))$,   then we can take  $t_0 = 0$ in \eqref{375} and immediately obtain uniqueness. However, such a regularity for $\u$ is not known to be true for NSE, nor even for solutions of the heat equation.

It has been proved in \cite[Theorem 3.2.1]{MDJC}  that if $$\u\in\mathrm{L}^2(0,T;\H^{\frac{d}{2}-1}(\Omega))\ \text{ and }\ \sqrt{t}\u\in \mathrm{L}^2(0,T;\H^{\frac{d}{2}+1}(\Omega)),$$ then the ODE \eqref{372} has a unique solution for every $\boldsymbol{X}_0\in\mathbb{T}^d$. One can weaken the assumption on $\u$ to $\u\in\mathrm{L}^p(0,T;\H^{\frac{d}{2}-1}(\Omega))$ for some $p>1$ (\cite{MDJC1}). Moreover, from \cite[Theorem 3.2.2]{MDJC}, we can obtain the continuity with respect to the initial data. 
Suppose that $\u_n\to \u$ strogly in $\mathrm{L}^2(0,T;\H^{\frac{d}{2}-1}(\Omega))$ and that $\sqrt{t}\u_n$ is bounded in $\mathrm{L}^2(0,T;\H^{\frac{d}{2}+1}(\Omega)).$ For some $\boldsymbol{X}_0\in\mathbb{T}^d$, let $\boldsymbol{X}_n(t)$ be the unique solution of
\begin{equation}
	{\boldsymbol{X}'_n}=\u_n(\boldsymbol{X}_n,t), \
		\boldsymbol{X}_n(0)=\boldsymbol{X}_0.
\end{equation}
Then $\boldsymbol{X}_n(t)\to\boldsymbol{X}(t)$ uniformly on $[0,T]$, where $\boldsymbol{X}(t)$  solves \eqref{372}. A similar result to the following theorem for 2D NSE in periodic domains is obtained in \cite{MDJC}, and in bounded domains it is established in \cite{MDJC1}.
\begin{theorem}\label{thm3.11}
	If 
	\begin{equation}
		\left\{
	\begin{aligned}
	&	\x\in {\L}^2_p(\T^2), \f\in \mathrm{L}^2_p(0,T; {\L}^2(\T^2))\ \text{ for } \ d=2,r\in[1,\infty),\\
	&	\x\in {\H}^1_p(\T^3), \f\in\mathrm{L}^2_p(0,T; {\H}^{\frac{1}{2}}(\T^3))\ \text{ for }\ d=3,r\in[3,5]\  (2\beta\mu\geq 1\ \text{ for }\ r=3), \\
	&\x\in {\H}_p^2(\T^3),\f\in\W^{1,2}(0,T; {\L}^2_p(\T^3))\cap\mathrm{L}^{2}(0,T; {\H}_p^{\frac{1}{2}}(\T^3))\ \text{ for }\ d=3, r\in(5,\infty),
	\end{aligned}
\right.
\end{equation}
and 	$\u(t)$ is the corresponding solution of the 2D and 3D CBF equations \eqref{1} on $[0, T ]$, respectively,  then the solution $\boldsymbol{X}(\cdot)$  of \eqref{372} is unique. Furthermore, for each fixed $\boldsymbol{X}_0\in\mathbb{T}^d,$ the map $\x \mapsto\boldsymbol{X}(\cdot)$   is continuous from $ {\H}_p^{d-2}(\mathbb{T}^d)$  ($ {\H}_p^{2}(\T^3)$ for $d=3$ and $r\in(5,\infty)$) into $\C([0, T ];\R^d).$
\end{theorem}
\begin{proof}
For  $d=2$, $r\in[1,\infty)$, $\x\in\H$ and $\f\in\mathrm{L}^2(0,T;\H)$, there exists a unique weak  solution $\u\in\mathrm{C}([0,T];\H)\cap\mathrm{L}^2(0,T;\V)\cap\mathrm{L}^{r+1}(0,T;\wi\L^{r+1})$ to the system \eqref{555}.  From the estimates \eqref{16} and \eqref{05},  it is clear that $\sqrt{t}\u\in\mathrm{L}^{\infty}(0,T;\V)\cap\mathrm{L}^2(0,T;\D(\A))$. Therefore,  we obtain $\sqrt{t}\u\in\mathrm{L}^2(0,T; {\H}^{2}(\mathbb{T}^2))$, since $\|\A\u\|_{\H}\leq C\|\u\|_{ {\H}_p^2}$. Hence \cite[Theorem 3.2.1]{MDJC} now guarantees the uniqueness of the Lagrangian trajectories corresponding to weak solutions of the equations in two dimensions. For the continuity with respect to the initial data, \cite[Theorem 3.2.2]{MDJC} requires uniform estimates for  $\sqrt{t}\u_n\in\mathrm{L}^2(0,T; {\H}_p^{2}(\mathbb{T}^2))$, when $\u_n(0)\to\x$  strongly in $ {\L}^2_p(\T^2)$. These follow immediately from \eqref{02} and \eqref{05} since $\u_n(0)$ is  uniformly bounded in $ {\L}^2_p(\T^2)$. The strong convergence of $\u_n\to\u$ in $\mathrm{L}^2 (0, T ;  {\L}^2_p(\T^2))$   follows along an appropriate subsequence  by an application of the Aubin-Lions compactness theorem (\cite[Theorem 8.1]{JCR_2001}) from the uniform bounds on $\u_n\in\mathrm{L}^2(0, T ;  {\H}_p^1(\T^2))$ and on $\frac{\d\u_n}{\d t}\in \mathrm{L}^{\frac{r+1}{r}}(0,T; {\H}_p^{-1}(\T^2))$.

For the case $d=3$, $r\in[3,5]$,  $\x\in\H$ and $\f\in\mathrm{L}^2(0,T; {\H}_p^{\frac{1}{2}}(\T^3))$, there exists a unique strong solution $\u\in\mathrm{C}([0,T];\V)\cap\mathrm{L}^2(0,T;\D(\A))\cap\mathrm{L}^{r+1}(0,T;\wi\L^{3(r+1)})$ to the system \eqref{555}. We need to show that $\sqrt{t}\u\in\mathrm{L}^2(0,T; {\H}_p^{\frac{5}{2}}(\T^d))$. Taking the inner product with $\A\u(\cdot)$  to the first equation in \eqref{555}  and calculations similar to \eqref{05}  and \eqref{016}  yield 
\begin{align}\label{3.78}
&	\|\nabla\u(t)\|_{\H}^2+\mu\int_0^t\|\A\u(s)\|_{\H}^2+\beta\int_0^t\||\u(s)|^{\frac{r-1}{2}}|\nabla\u(s)|\|_{\H}^2\d s\leq \|\nabla\x\|_{\H}^2+\frac{2\vartheta K}{\mu^2}+\frac{2}{\mu}\int_0^T\|\f(t)\|_{\H}^2\d t,
\end{align}
for all $t\in[0,T]$, where $K$ is defined in \eqref{02}. Taking the inner product with $t\A^{\frac{3}{2}}\u$ to the first equation in \eqref{555}, we find 
\begin{align}\label{3p80}
	&	\frac{1}{2}\frac{\d}{\d t}\left(t\|\A^{\frac{3}{4}}\u(t)\|_{\H}^2\right)+\mu t \|\A^{\frac{5}{4}}\u(t)\|_{\H}^2+\alpha t \|\A^{\frac{3}{4}}\u(t)\|_{\H}^2\nonumber\\&=\frac{1}{2}\|\A^{\frac{3}{4}}\u(t)\|_{\H}^2+(\f(t),t\A^{\frac{3}{2}}\u(t))-(\B(\u(t)),t\A^{\frac{3}{2}}\u(t))-\beta(\mathcal{C}(\u(t)),t\A^{\frac{3}{2}}\u(t)),
\end{align}
for a.e. $t\in[0,T]$. We estimate $|(\f,\A^{\frac{3}{2}}\u)|$ using H\"older's and Young's inequalities as 
\begin{align}\label{381}
	|(\f,\A^{\frac{3}{2}}\u)|\leq\|\A^{\frac{1}{4}}\f\|_{\H}\|\A^{\frac{5}{4}}\u\|_{\H}\leq\frac{\mu}{4}\|\A^{\frac{5}{4}}\u\|_{\H}^2+\frac{1}{\mu}\|\A^{\frac{1}{4}}\f\|_{\H}^2. 
\end{align}
We estimate $|(\B(\u),\A^{\frac{3}{2}}\u)|$  using fractional Leibniz rule (\cite[Theorem 1]{LGSO}), H\"older's, Agmon's and Young's inequalities as 
\begin{align}\label{3p82}
	|(\B(\u),\A^{\frac{3}{2}}\u)|&\leq \|\A^{\frac{1}{4}}\B(\u)\|_{\H}\|\A^{\frac{5}{4}}\u\|_{\H}\leq C\|\A^{\frac{3}{4}}(\u\otimes\u)\|_{\H}\|\A^{\frac{5}{4}}\u\|_{\H}\nonumber\\&\leq C\|\u\|_{\wi\L^{\infty}}\|\A^{\frac{3}{4}}\u\|_{\H}\|\A^{\frac{5}{4}}\u\|_{\H}\leq C\|\u\|_{\V}^{\frac{1}{2}}\|\u\|_{\H^2}^{\frac{1}{2}}\|\A^{\frac{3}{4}}\u\|_{\H}\|\A^{\frac{5}{4}}\u\|_{\H}\nonumber\\&\leq \frac{\mu}{8}\|\A^{\frac{5}{4}}\u\|_{\H}^2+\frac{C}{\mu}\|\u\|_{\V}\|\u\|_{\H^2}\|\A^{\frac{3}{4}}\u\|_{\H}^2. 
\end{align}
Similarly, we estimate $\beta|(\mathcal{C}(\u),\A^{\frac{3}{2}}\u)|$ using Theorem A.6, \cite{CEK}, H\"older's  Gagliardo-Nirenberg's and Young's inequalities as 
\begin{align}\label{383}
	\beta|(\mathcal{C}(\u),\A^{\frac{3}{2}}\u)|&\leq \beta\|\A^{\frac{1}{4}}\mathcal{C}(\u)\|_{\H}\|\A^{\frac{5}{4}}\u\|_{\H}\leq C\beta\|\u\|_{\wi\L^{3(r-1)}}^{r-1}\|\A^{\frac{1}{4}}\u\|_{\wi\L^{6}}\|\A^{\frac{5}{4}}\u\|_{\H}\nonumber\\&\leq C\beta\|\u\|_{\wi\L^{r+1}}\|\u\|_{\wi\L^{3(r+1)}}^{(r-2)}\|\A^{\frac{3}{4}}\u\|_{\H}\|\A^{\frac{5}{4}}\u\|_{\H}\nonumber\\&\leq\frac{\mu}{8}\|\A^{\frac{5}{4}}\u\|_{\H}^2+\frac{C\beta^2}{\mu}\|\u\|_{\wi\L^{r+1}}^{2}\|\u\|_{\wi\L^{3(r+1)}}^{2(r-2)}\|\A^{\frac{3}{4}}\u\|_{\H}^2.
\end{align}
Combining \eqref{381}-\eqref{383} and substituting it in \eqref{3p80}, we deduce 
\begin{align}\label{384}
&	t\|\A^{\frac{3}{4}}\u(t)\|_{\H}^2+\mu \int_0^ts\|\A^{\frac{5}{4}}\u(s)\|_{\H}^2\d s+2\alpha \int_0^ts \|\A^{\frac{3}{4}}\u(s)\|_{\H}^2\d s\nonumber\\&\leq \int_0^t\|\u(s)\|_{\V}\|\A\u(s)\|_{\H}\d s+\frac{2}{\mu}\int_0^ts\|\A^{\frac{1}{4}}\f(s)\|_{\H}^2\d s+\frac{C}{\mu}\int_0^ts\|\u(s)\|_{\V}\|\u(s)\|_{\H^2}\|\A^{\frac{3}{4}}\u(s)\|_{\H}^2\d s\nonumber\\&\quad+\frac{C\beta^2}{\mu}\int_0^ts\|\u(s)\|_{\wi\L^{r+1}}^{2}\|\u(s)\|_{\wi\L^{3(r+1)}}^{2(r-2)}\|\A^{\frac{3}{4}}\u(s)\|_{\H}^2\d s,
\end{align}
for all $t\in[0,T]$. An application of Growall's inequality in \eqref{384} yields 
\begin{align}\label{385}
	&	t\|\A^{\frac{3}{4}}\u(t)\|_{\H}^2+\mu \int_0^ts\|\A^{\frac{5}{4}}\u(s)\|_{\H}^2\d s\nonumber\\&\leq \left\{T^{\frac{1}{2}}\sup\limits_{t\in[0,T]}\|\u(t)\|_{\V}\left(\int_0^T\|\A\u(t)\|_{\H}^2\d t\right)^{\frac{1}{2}}+\frac{2T}{\mu}\int_0^T\|\A^{\frac{1}{4}}\f(t)\|_{\H}^2\d t\right\}\nonumber\\&\qquad\times\exp\left\{\frac{CT^{\frac{1}{2}}}{\mu}\sup\limits_{t\in[0,T]}\|\u(t)\|_{\V}\left(\int_0^T\|\u(t)\|_{\H^2}^2\d t\right)^{\frac{1}{2}}\right\}\nonumber\\&\qquad\times\exp\left\{\frac{C\beta^2}{\mu}T^{\frac{5-r}{r+1}}\sup_{t\in[0,T]}\|\u(t)\|_{\V}^{2}\left(\int_0^T\|\u(t)\|_{\wi\L^{3(r+1)}}^{r+1}\d t\right)^{\frac{2(r-2)}{r+1}}\right\},
\end{align}
for all $t\in[0,T]$ and $r\in[3,5]$. From the estimate \eqref{3.78}, it is clear that $\sqrt{t}\u\in\mathrm{L}^2(0,T;\D(\A^{\frac{5}{4}}))$, as required. Standard energy estimate (see \eqref{02}) shows that $\u_n$ is uniformly bounded in $\mathrm{L}^2(0, T ; {\H}_p^1(\mathbb{T}^3))$, and $\frac{\d\u_n}{\d t}$ is uniformly bounded in $\mathrm{L}^{\frac{r+1}{r}}(0, T ;  {\H}_p^{-1}(\mathbb{T}^3)+\L^{\frac{r+1}{r}}(\T^d))$ (this is true even if $\x\in {\L}_p^2(\T^3)$). Since for $0\leq s<1$, $ {\H}_p^1(\T^3)\subset {\H}_p^s(\T^3)\subset {\H}_p^{-1}(\mathbb{T}^3)+\L^{\frac{r+1}{r}}(\T^d))$ and the first embedding is compact, it follows from the Aubin-Lions compactness theorem that $\u_n\to\u$ strongly in $\mathrm{L}^2(0, T ;  {\H}_p^s(\T^3))$ for any $s < 1$. In particular $\u_n \to \u$ strongly in $\mathrm{L}^2(0, T ;  {\H}_p^{\frac{1}{2}}(\T^3))$.  Thus, we have  the continuity with respect to the initial data by applying \cite[Theorem 3.2.1]{MDJC}.

For $d=3$, $r\in(5,\infty)$, 
taking the inner product with $\frac{\d\u}{\d t}$ to the first equation in \eqref{555} and calculations similar to \eqref{5p48} and \eqref{09} results to 
\begin{align}
&	\mu \|\nabla\u(t)\|_{\H}^2+\alpha\|\u(t)\|_{\H}^2+\frac{2\beta}{r+1}\|\u(t)\|_{\wi\L^{r+1}}^{r+1}+\int_0^t\left\|\frac{\d\u(s)}{\d t}\right\|_{\H}^2\d s\nonumber\\&\leq 2\mu\|\nabla\x\|_{\H}^2+\alpha\|\x\|_{\H}^2+\frac{2\beta}{r+1}\|\x\|_{\wi\L^{r+1}}^{r+1}+\frac{4\vartheta K}{\mu}+2\int_0^T\|\f(t)\|_{\H}^2\d t,
	\end{align}
for all $t\in[0,T]$. Taking the inner product with $\v(\cdot)$ to the first equation in \eqref{324} and calculations similar to \eqref{327} provide 
\begin{align}\label{380}
	\|\v(t)\|_{\H}^2&\leq \left\{\|\v(0)\|_{\H}^2+\frac{2}{\mu}\int_0^T\|\f_t(t)\|_{\V'}^2\d t\right\}e^{\frac{2\vartheta KT}{\mu}}\nonumber\\&\leq \bigg\{\left(\mu\|\A\x\|_{\H}+\|\x\|_{\V}^{\frac{3}{2}}\|\A\x\|_{\H}^{\frac{1}{2}}+\alpha\|\x\|_{\H}+\beta\|\x\|_{\wi\L^{2r}}^{r}+\|\f(0)\|_{\H}\right)^2\nonumber\\&\qquad+\frac{2}{\mu\lambda_1}\int_0^T\|\f_t(t)\|_{\H}^2\d t\bigg\}e^{\frac{2\vartheta KT}{\mu}}=K_1,
\end{align}
for all $t\in[0,T]$, where $\v=\frac{\d\v}{\d t}$. Note that $\f\in\W^{1,2}(0,T;\H)$ implies $\f\in\C([0,T];\H)$ also.  It follows from \eqref{555} that 
\begin{align}\label{3p78}
	&\mu \|\A\u\|_{\H}^2+\alpha \|\nabla\u\|_{\H}^2+\beta \||\u|^{\frac{r-1}{2}}|\nabla\u|\|_{\H}^2+4\beta \left[\frac{(r-1)}{(r+1)^2}\right]\|\nabla|\u|^{\frac{r+1}{2}}\|_{\H}^2\nonumber\\&=-\left(\frac{\d\u}{\d t},\A\u\right)-(\B(\u),\A\u)\leq\left\|\frac{\d\u}{\d t}\right\|_{\H}\|\A\u\|_{\H}+\|\u\|_{\wi\L^{\infty}}\|\u\|_{\V}\|\A\u\|_{\H}
	\nonumber\\ & \leq\left\|\frac{\d\u}{\d t}\right\|_{\H}\|\A\u\|_{\H}+  \|\u\|^{\frac12}_{\V}\|(\I+\A)\u\|^{\frac12}_{\H}\|\u\|_{\V}\|\A\u\|_{\H}
	\nonumber\\&\leq\frac{\mu }{2}\|\A\u\|_{\H}^2+\frac{C }{\mu}\|\u\|_{\V}^6+\frac{1}{\mu}\left\|\frac{\d\u}{\d t}\right\|_{\H}^2.
\end{align}
Therefore, we deduce from \eqref{3.78} and \eqref{380} that 
\begin{align}\label{382}
	&\mu\|\A\u(t)\|_{\H}^2+\alpha \|\nabla\u\|_{\H}^2+\beta \|\u(t)\|_{\wi\L^{3(r+1)}}^{r+1}\leq \frac{C}{\mu}\left\{\|\x\|_{\V}^2+\frac{2\vartheta K}{\mu^2}+\frac{2}{\mu}\int_0^T\|\f(t)\|_{\H}^2\d t\right\}^3+\frac{K_1}{\mu},
\end{align}
for all $t\in[0,T]$. Thus from \eqref{384} and  \eqref{382}, one can conclude the proof in this case also. 
\end{proof}

		\section{The Stochastic Case}\label{Sec5}\setcounter{equation}{0}
	Let $(\Omega,\mathscr{F},\mathbb{P})$ be a complete probability space equipped with an increasing family of sub-sigma fields $\{\mathscr{F}_t\}_{0\leq t\leq T}$ of $\mathscr{F}$ satisfying:
	\begin{enumerate}
		\item [(i)] $\mathscr{F}_0$ contains all elements $F\in\mathscr{F}$ with $\mathbb{P}(F)=0$,
		\item [(ii)] $\mathscr{F}_t=\mathscr{F}_{t+}=\bigcap\limits_{s>t}\mathscr{F}_s,$ for $0\leq t\leq T$.
	\end{enumerate} 
Let $\{\W(t)\}_{t\in[0,T]}$ be a one-dimensional real-valued Brownian motion on $(\Omega,\mathscr{F},\{\mathscr{F}_t\}_{0\leq t\leq T},\mathbb{P}).$ 
	We consider the following stochastic CBF equations with a linear multiplicative noise: 
		\begin{equation}\label{36}
		\left\{
		\begin{aligned}
			\d\u(t)+[\mu \A\u(t)+\B(\u(t))+\alpha\u(t)+\beta\mathcal{C}(\u(t))]\d t&=\f(t)
			\d t+\sigma\u(t)\d\W(t), \\
			\u(0)&=\boldsymbol{x},
		\end{aligned}\right. 
	\end{equation} 
where  $\sigma\in\mathbb{R}\backslash\{0\}$ and $\f\in\mathrm{L}^2(0,T;\H)$ is a deterministic external forcing. Note that $z(t)=e^{-\sigma\W(t)} \in\C([0,T];\mathbb{R}),$ $\mathbb{P}$-a.s.  satisfies 
\begin{equation*}
	\left\{
\begin{aligned}
	\d z(t)&=-\sigma z(t)\d\W(t)+\frac{\sigma^2}{2}z(t)\d t, \\ z(0)&=1.
\end{aligned}
\right.
\end{equation*}
Using the transformation $\v(t)=\u(t)z(t)$, we obtain the following random dynamical system: 
	\begin{equation}\label{37}
	\left\{
	\begin{aligned}
		\frac{\d\v(t)}{\d t}+\mu \A\v(t)+\frac{1}{z(t)}\B(\v(t))+\left(\alpha+\frac{\sigma^2}{2}\right)\v(t)+\frac{\beta}{[z(t)]^{r-1}}\mathcal{C}(\v(t))&=z(t)\boldsymbol{f}(t), \\
		\v(0)&=\boldsymbol{x}.
	\end{aligned}\right. 
\end{equation}
As $z(t)\in\C([0,T];\mathbb{R}),$ $\mathbb{P}$-a.s.,  for $\x\in\H$ and $\f\in\mathrm{L}^2(0,T;\H)$, the existence and uniqueness of weak solution to the system \eqref{37} can be proved in a similar way as in \cite{SNA,KWH,MT1}. Our aim is to prove the pathwise backward uniqueness of the system \eqref{36}. 

\begin{theorem}[Backward uniqueness]\label{thm4.1}
		Let $\x\in\H$, $\f\in\mathrm{L}^{2}(0,T;\H)$ and  $\u_1,\u_2$ satisfy the first equation in the system \eqref{36}. For $d=2,r\in[1,\infty)$ and $d=3,r\in[3,5]$ ($2\beta\mu\geq 1$ for $d=r=3$), if $\u_1(T)=\u_2(T)$ in $\H$, then $\u_1(t)=\u_2(t)$ in $\H,$ $\mathbb{P}$-a.s., for all $t\in[0,T].$ 
\end{theorem} 
\begin{proof}
Since $\v(t)=\u(t)z(t)$, it is enough to show that if $\v_1(T)=\v_2(T)$ in $\H$, then $\v_1(t)=\v_2(t)$ in $\H$  for all $t\in[0,T].$ Taking the inner product with $\v(\cdot)$ to the first equation in \eqref{37} and then integrating from $0$ to $T$, we find 
\begin{align*}
	&\|\v(t)\|_{\H}^2+2\mu\int_0^t\|\nabla\v(s)\|_{\H}^2\d s+2\left(\alpha+\frac{\sigma^2}{2}\right)\int_0^t\|\v(s)\|_{\H}^2\d s+2\beta\int_0^te^{\sigma(r-1)\W(s)}\|\v(s)\|_{\wi\L^{r+1}}^{r+1}\d s\nonumber\\&=\|\x\|_{\H}^2+2\int_0^t\langle z(s)\f(s),\v(s)\rangle\d s
	\nonumber\\ & \leq\|\x\|_{\H}^2+\min\{\mu,\alpha\}\int_0^t\|\v(s)\|_{\V}^2\d s+\frac{1}{\min\{\mu,\alpha\}}\sup_{s\in[0,t]}[z(s)]^2\int_0^t\|\f(s)\|_{\V'}^2\d s,
\end{align*}
for all $t\in[0,T]$. Therefore, we get 
\begin{align}\label{55}
	&\|\v(t)\|_{\H}^2+\mu\int_0^t\|\nabla\v(s)\|_{\H}^2\d s+\left(\alpha+\sigma^2\right)\int_0^t\|\v(s)\|_{\H}^2\d s+2\beta\int_0^te^{\sigma(r-1)\W(s)}\|\v(s)\|_{\wi\L^{r+1}}^{r+1}\d s\nonumber\\&\leq \|\x\|_{\H}^2+\frac{1}{\mu}\sup_{t\in[0,T]}e^{-2\sigma\W(t)}\int_0^T\|\f(t)\|_{\V'}^2\d t=\widetilde{K}.
\end{align}
Also note that 
\begin{align*}
	\int_0^T\|\v(t)\|_{\wi\L^{r+1}}^{r+1}\d t\leq\sup_{t\in[0,T]}e^{-\sigma(r-1)\W(t)}\int_0^Te^{\sigma(r-1)\W(t)}\|\v(t)\|_{\wi\L^{r+1}}^{r+1}\d t\leq\frac{C\widetilde{K}}{2\beta}.
\end{align*}
Moreover, one can show that $\v\in\mathrm{C}([0,T];\H)\cap\mathrm{L}^2(0,T;\V)\cap\mathrm{L}^{r+1}(0,T;\wi\L^{r+1})$.  	For $d=2,3$ and $r\in(3,\infty)$, calculations similar to \eqref{16}, \eqref{05} and \eqref{18} yield 
	\begin{align}\label{56}
&	\|\nabla\v(t)\|_{\H}^2+	\mu\int_{\epsilon}^t\|\A\v(s)\|_{\H}^2\d s+\beta	\int_{\epsilon}^te^{\sigma(r-1)\W(s)}\|\v(s)\|_{\wi\L^{3(r+1)}}^{r+1}\d s\nonumber\\&\leq \left\{\begin{array}{lc}C\left(\frac{\widetilde{K}}{2\mu  \epsilon}+\frac{2}{\mu}\sup\limits_{t\in[0,T]}e^{-2\sigma\W(t)}\int_0^T\|\f(t)\|_{\H}^2\d t\right),&\text{ for }d=2,\\
	C\left(	\frac{\widetilde{K}}{\mu \epsilon}+\frac{4\widetilde{K}\vartheta}{\mu^2}+\frac{4}{\mu}\sup\limits_{t\in[0,T]}e^{-2\sigma\W(t)}\int_0^T\|\f(t)\|_{\H}^2\d t\right),&\text{ for }d=3,\end{array}\right.
	\end{align}
for all $t\in[\epsilon,T]$, for any $\epsilon>0$. Therefore, we have  $\v\in\C((0,T];\V)\cap\mathrm{L}^2(\epsilon,T;\D(\A))\cap\mathrm{L}^{r+1}(\epsilon,T;\wi\L^{3(r+1)})$ for any $\epsilon>0$ by showing an estimate similar to \eqref{010}. 

Let us now prove the backward uniqueness result for the system \eqref{36}. Let $\v_1(\cdot)$ and $\v_2(\cdot)$ be two solutions of the system \eqref{37} with the same final data, say $\boldsymbol{\xi}$ and external forcing $\f\in\mathrm{L}^2(0,T;\H)$.  Then $\v=\v_1-\v_2$ satisfies the following system in $\V'+\wi\L^{\frac{r+1}{r}}$ for a.e. $t\in[0,T]$ and in $\H$ for a.e. $t\in[\epsilon,T]$, for any $0<\epsilon<T$:
\begin{align}
	\left\{
	\begin{aligned}
		\frac{\d  \v(t)}{\d  t}+\widehat{\mathcal{A}}\v(t)&=-\frac{1}{z(t)}[\B(\v_1(t),\v(t))+\B(\v(t),\v_2(t))]\\&\quad+\frac{\beta}{[z(t)]^{r-1}}\int_0^1\mathcal{C}'(\theta\v_1(t)+(1-\theta)\v_2(t))\d\theta\v(t)\bigg]=\widehat{h}(\v(t)),\\
		\v(T)&=\boldsymbol{0},
	\end{aligned}\right. 
\end{align}
where 
\begin{align}
	\widehat{\mathcal{A}}\v:=\mu\A\v+\left(\alpha+\frac{\sigma^2}{2}\right)\v.
\end{align}
Our aim is to show that if $\v(T)=\boldsymbol{0}$, then $\v(t)=\boldsymbol{0}$ for all $t\in[0,T]$. We prove this result by a contradiction first in the interval $[\epsilon,T],$ for any $\epsilon>0$ and then by using  the continuity of $\|\u(t)\|_{\H}$ in $[0,T]$ and  the arbitrariness of $\epsilon$, one can obtain the required result in $[0,T]$. Assume that there exists some $t_0\in[\epsilon,T)$ such that $\v(t_0)\neq \boldsymbol{0}$. Since the mapping $t\mapsto\|\v(t)\|_{\H}$ is continuous, the following alternative holds: 
\begin{enumerate}
	\item [(i)] for all $t\in[t_0,T]$, $\|\v(t)\|_{\H}>0$ or
	\item [(ii)] there exists a $t_1\in(t_0,T)$ such that for all $t\in(t_0,t_1)$, $\|\v(t)\|_{\H}>0$ and $\v(t_1)=\boldsymbol{0}$. 
\end{enumerate}
In the second case, denoting by $\widehat{\Lambda}(t)$, the ratio 
\begin{align}\label{58}
	\widehat{\Lambda}(t)=\frac{\langle\widehat{\mathcal{A}}\v(t),\v(t)\rangle}{\|\v(t)\|_{\H}^2}=\frac{\mu\|\v(t)\|_{\V}^2+\left(\alpha+\frac{\sigma^2}{2}\right)\|\v(t)\|_{\H}^2}{\|\v(t)\|_{\H}^2}\geq \min\{\mu,\alpha\}\frac{\|\v(t)\|_{\V}^2}{\|\v(t)\|_{\H}^2}.
\end{align}
Calculations similar to \eqref{33} and \eqref{016} provide
\begin{align}\label{59}
	\frac{1}{2}\frac{\d\widehat{\Lambda}}{\d t}+\frac{\|\widehat{\mathcal{A}}\v-\widehat{\Lambda}\v\|_{\H}^2}{\|\v\|_{\H}^2}&=\frac{\langle\widehat{\mathcal{A}}\v-\widehat{\Lambda}\v,\widehat{h}(\v)\rangle}{\|\v\|_{\H}^2}\leq\frac{1}{2}\frac{\|\widehat{\mathcal{A}}\v-\widehat{\Lambda}\v\|_{\H}^2}{\|\v\|_{\H}^2}+\frac{1}{2}\frac{\|\widehat{h}(\v)\|_{\H}^2}{\|\v\|_{\H}^2}. 
\end{align} 
For $d=2$ and $r\in[1,\infty)$, one can use the estimates given in \eqref{018} and \eqref{47} to obtain an estimate for $\|\widehat{h}(\v)\|_{\H}^2$ . Thus, we consider the case $d=3$ and $r\in[3,5]$ ($2\beta\mu\geq 1$ for $d=3$). We estimate $\|\widehat{h}(\v)\|_{\H}^2$ using \eqref{019} and \eqref{47}  as 
\begin{align}\label{49}
	\|\widehat{h}(\v)\|_{\H}^2&\leq\frac{C}{[z(t)]^2}\left(\|\v_1\|_{\V}\|\v_1\|_{\H^2}+C\|\v_2\|_{\V}\|\v_2\|_{\H^2}\right)\|\v\|_{\V}^2\nonumber\\&\quad+\frac{C\beta}{[z(t)]^{2(r-1)}}\left(\|\v_1\|_{\wi\L^{r+1}}^{\frac{r-1}{2}}\|\v_1\|_{\wi\L^{3(r+1)}}^{\frac{3(r-1)}{2}}+\|\v_2\|_{\wi\L^{r+1}}^{\frac{r-1}{2}}\|\v_2\|_{\wi\L^{3(r+1)}}^{\frac{3(r-1)}{2}}\right)\|\v\|_{\V}^2.
\end{align}
Therefore, from \eqref{58} and \eqref{59}, we have 
\begin{align}
	\frac{\d\widehat{\Lambda}}{\d t}&\leq\frac{C}{\mu[z(t)]^2}\left(\|\v_1\|_{\V}\|\v_1\|_{\H^2}+C\|\v_2\|_{\V}\|\v_2\|_{\H^2}\right)\widehat{\Lambda}\nonumber\\&\quad+\frac{C\beta}{\mu[z(t)]^{2(r-1)}}\left(\|\v_1\|_{\wi\L^{r+1}}^{\frac{r-1}{2}}\|\v_1\|_{\wi\L^{3(r+1)}}^{\frac{3(r-1)}{2}}+\|\v_2\|_{\wi\L^{r+1}}^{\frac{r-1}{2}}\|\v_2\|_{\wi\L^{3(r+1)}}^{\frac{3(r-1)}{2}}\right)\widehat{\Lambda}. 
\end{align}
An application of the variation of constants formula yields 
\begin{align}\label{63}
	\widehat{\Lambda}(t)&\leq\widehat{\Lambda}(t_0)\exp\left[\frac{CT^{\frac{1}{2}}}{\mu}\sup\limits_{t\in[0,T]}e^{2\sigma\W(t)}\sum\limits_{i=1}^2\sup\limits_{t\in[t_0,T]}\|\v_i(t)\|_{\V}\left(\int_{t_0}^T\|\v_i(t)\|_{\H^2}^2\d t\right)^{\frac{1}{2}}\right]\nonumber\\&\quad\times\exp\left[\frac{C\beta T^{\frac{5-r}{2(r+1)}}}{\mu}\sup\limits_{t\in[0,T]}e^{2(r-1)\sigma\W(t)}\sum_{i=1}^2\sup\limits_{t\in[t_0,T]}\|\v_i(t)\|_{\V}^{\frac{r-1}{2}}\left(\int_{t_0}^T\|\v_i(t)\|_{\wi\L^{3(r+1)}}^{r+1}\d t\right)^{\frac{3(r-1)}{2(r+1)}}\right].
\end{align}
Using the estimate \eqref{56},  one can easily see that the right hand side of \eqref{63} is finite. 

On the other hand, we infer
\begin{align}\label{64}
	-\frac{1}{2\|\v(t)\|_{\H}^2}\frac{\d}{\d t}\|\v(t)\|_{\H}^2=-\frac{\langle\v(t),\partial_t\v(t)\rangle}{\|\v(t)\|_{\H}^2}=\widehat{\Lambda}(t)-\frac{\langle \widehat{h}(\v(t)),\v(t)\rangle}{\|\v(t)\|_{\H}^2}. 
\end{align}
For $d=2$, $r\in[1,\infty)$, we can use calculations similar to \eqref{43} and \eqref{48} to estimate $\frac{|\langle \widehat{h}(\v),\v\rangle|}{\|\v\|_{\H}^2}$. Thus, we consider the case $d=3$ and $r\in[3,5]$ ($2\beta\mu\geq 1$ for $d=3$).  Calculations similar to \eqref{43} and \eqref{48} yield 
\begin{align}\label{65}
	\frac{|\langle \widehat{h}(\v),\v\rangle|}{\|\v\|_{\H}^2}&\leq\frac{2^{1/2}\left(\|\v_1\|_{\wi\L^4}+\|\v_2\|_{\wi\L^4}\right)\|\v\|_{\V}^{7/4}}{\|\v\|_{\H}^{7/4}}+\frac{C}{\|\u\|_{\H}^{\frac{2r}{r+1}}}\left(\|\u_1\|_{\wi\L^{r+1}}^{r-1}+\|\u_2\|_{\wi\L^{r+1}}^{r-1}\right)\|\u\|_{\V}^{\frac{2r}{r+1}}\nonumber\\&\leq\widehat{\Lambda}+\frac{C}{\mu^7}\left(\|\v_1\|_{\wi\L^4}^8+\|\v_2\|_{\wi\L^4}^8\right)+\frac{C}{\mu^{r+1}}\left(\|\v_1\|_{\wi\L^{r+1}}^{(r+1)(r-1)}+\|\v_2\|_{\wi\L^{r+1}}^{(r+1)(r-1)}\right).
\end{align}
Therefore, using \eqref{65} in  \eqref{64}, we easily have 
\begin{align}\label{66}
	-&\frac{\d}{\d t}\log\|\v(t)\|_{\H}\nonumber\\&\leq 2\widehat{\Lambda}(t)+\frac{C}{\mu^7}\left(\|\v_1(t)\|_{\V}^8+\|\v_2(t)\|_{\V}^8\right)+\frac{C}{\mu^{r+1}}\left(\|\v_1(t)\|_{\V}^{(r+1)(r-1)}+\|\v_2(t)\|_{\V}^{(r+1)(r-1)}\right).
\end{align}
According to \eqref{63}, \eqref{55} and \eqref{56}, the right hand side of \eqref{66} is integrable on $(t_0,t_1)$ and this contradicts the fact that $\v(t_1)=\boldsymbol{0}$. Thus we are in the case (i) of the alternative and backward uniqueness result follows. 
\end{proof}

\begin{corollary}
	Under the assumption of Theorem \ref{thm4.1},  either  $\u$ vanishes identically or $\u$ 	never vanishes. 
\end{corollary}

For $d=2,r\in[1,\infty)$ and $d=3,r\in[3,5]$ ($2\beta\mu\geq 1$ for $d=r=3$),  the following result is a direct consequence of the backward uniqueness result which can be proved in a similar way as in Theorem \ref{thm3.3}.

\begin{theorem}[Approximate controllability]\label{thm4.2}
	The space $\{\u^{\boldsymbol{x}}(T):\boldsymbol{x}\in\H\}$ is dense in $\H$, $\mathbb{P}$-a.s., where $\u^{\boldsymbol{x}}(\cdot)$ is the unique solution to the  system \eqref{36}. 
\end{theorem}

 \medskip\noindent
{\bf Acknowledgments:} K. Kinra is funded by national funds through the FCT - Funda\c c\~ao para a Ci\^encia e a Tecnologia, I.P., under the scope of the project  UIDB/00297/2020 and UIDP/00297/ 2020 (Center for Mathematics and Applications). M. T. Mohan would  like to thank the Department of Science and Technology -Science and Engineering Research Board (DST-SERB), India for a Mathematical Research Impact Centric Support (MATRICS), File No.: MTR/2021/000066, grant. The author would also like to thank Prof. E. Zuazua, Friedrich-Alexander Universit\"at Erlangen–N\"urnberg for useful discussions.

\medskip\noindent	\textbf{Declarations:} 

\noindent 	\textbf{Ethical Approval:}   Not applicable 


\noindent  \textbf{Conflict of interest: }On behalf of all authors, the corresponding author states that there is no conflict of interest.

\noindent 	\textbf{Authors' contributions: } All authors have contributed equally. 

\noindent 	\textbf{Funding: } FCT - Portugal, UIDB/00297/2020 and UIDP/00297/ 2020 (K. Kinra), DST-SERB, India, MTR/2021/000066 (M. T. Mohan).

\noindent 	\textbf{Availability of data and materials: } Not applicable.

\end{document}